\documentclass[11pt,letterpaper]{amsart}
\usepackage[utf8]{inputenc}
\usepackage{amssymb, amsmath, enumerate, amsthm, cite, colonequals, stmaryrd, lipsum, graphicx, fancyvrb,fancyhdr, indentfirst, times, setspace}
\usepackage[margin=1in]{geometry}
\makeatletter
\@namedef{subjclassname@2020}{\textup{2020} Mathematics Subject Classification}
\makeatother

\usepackage[colorlinks]{hyperref}
\hypersetup{
 colorlinks=true,
 linkcolor=blue,
 filecolor=magenta, 
 urlcolor=cyan,
}

\PassOptionsToPackage{hyphens}{url}\usepackage{hyperref}

\usepackage[capitalise]{cleveref}
\crefformat{equation}{(#2#1#3)}
\crefrangeformat{equation}{(#3#1#4--#5#2#6)}
\crefformat{enumi}{(#2#1#3)}
\crefrangeformat{enumi}{(#3#1#4--#5#2#6)}

\usepackage[pagewise]{lineno}
\let\oldequation\equation
\let\oldendequation\endequation
\renewenvironment{equation}{\linenomathNonumbers\oldequation}{\oldendequation\endlinenomath}
\expandafter\let\expandafter\oldequationstar\csname equation*\endcsname
\expandafter\let\expandafter\oldendequationstar\csname endequation*\endcsname
\renewenvironment{equation*}{\linenomathNonumbers\oldequationstar}{\oldendequationstar\endlinenomath}
\let\oldalign\align
\let\oldendalign\endalign
\renewenvironment{align}{\linenomathNonumbers\oldalign}{\oldendalign\endlinenomath}
\expandafter\let\expandafter\oldalignstar\csname align*\endcsname
\expandafter\let\expandafter\oldendalignstar\csname endalign*\endcsname
\renewenvironment{align*}{\linenomathNonumbers\oldalignstar}{\oldendalignstar\endlinenomath}

\usepackage{tikz-cd}




\newcounter{intro}
\newcounter{result}

\newtheorem{introthm}[intro]{Theorem}

\newtheorem{introcor}[intro]{Corollary}

\newtheorem{introconjecture}[intro]{Conjecture}

\theoremstyle{definition}
\newtheorem{theorem}[result]{Theorem}
\newtheorem{lemma}[result]{Lemma}
\newtheorem{proposition}[result]{Proposition}
\newtheorem{corollary}[result]{Corollary}
\newtheorem{definition}[result]{Definition}
\newtheorem{example}[result]{Example}
\newtheorem{remark}[result]{Remark}

\newtheorem*{ack}{Acknowledgements}

\newtheorem*{addendum}{Addendum}

\newtheorem*{notation}{Notation}

\numberwithin{result}{section}
\numberwithin{equation}{subsection}

\title[Descent and generation]{Descent conditions for generation in derived categories}

\author[P.~Lank]{Pat Lank}
\address{P.~Lank,
Department of Mathematics,
University of South Carolina, 
Columbia, SC 29208,
U.S.A.}
\email{plankmathematics@gmail.com}

\date{\today}

\keywords{derived categories, strong generator, classical generator, Rouquier dimension, descent conditions, quasiexcellent schemes, derived splinters, resolution of singularities}

\subjclass[2020]{14F08 (primary), 14A30, 13D09, 18G80} 

\begin{document}

\begin{abstract}
    This work establishes a condition that determines when strong generation in the bounded derived category of a Noetherian $J\textrm{-}2$ scheme is preserved by the derived pushforward of a proper morphism. Consequently, we can produce upper bounds on the Rouquier dimension of the bounded derived category, and applications concerning affine varieties are studied. In the process, a necessary and sufficient constraint is observed for when a tensor-exact functor between rigidly compactly generated tensor triangulated categories preserves strong $\oplus$-generators.
\end{abstract}

\maketitle
\setcounter{tocdepth}{2}
\numberwithin{equation}{section}

\section{Introduction}
\label{sec:intro}

The derived category of bounded complexes with coherent cohomology on a scheme $X$ is denoted $D^b_{\operatorname{coh}}(X)$. This triangulated category has been an ignition for activity in recent years and in particular for algebraic geometry \cite{Orlov:2005,Bridgeland:2002,Kuznetsov:2021,Ballard/Favero/Katzarkov:2019}. There are powerful methodologies that strive to understand the structure of $D^b_{\operatorname{coh}}(X)$ \cite{Bondal/Orlov:1995,Okawa:2011, Alexeev/Orlov:2013,Kawamata:2006, Bridgeland:2002,Bridgeland:2007,Perry/Pertusi/Zhao:2022}. Amongst these closely related to our work are semiorthogonal decompositions which have produced a framework for treating geometric operations in the minimal model program as categorical transformations on $D^b_{\operatorname{coh}}(X)$ \cite{Kuznetsov:2021,Kuznetsov/Perry:2021,Kuznetsov:2007}. A major open problem in this area is the Bondal-Orlov localization conjecture \cite{Efimov:2020,Bondal/Orlov:2002}.

\begin{introconjecture}\label{conj:bondal_orlov_localization}
    Suppose $X$ is a variety over a field. If $\pi\colon \widetilde{X} \to X$ is a resolution of singularities and $X$ rational singularities, then the functor $\mathbb{R}\pi_\ast \colon D^b_{\operatorname{coh}}(\widetilde{X}) \to D^b_{\operatorname{coh}}(X)$ is a Verdier localization.
\end{introconjecture}

Roughly speaking, the functor $\mathbb{R}\pi_\ast \colon D^b_{\operatorname{coh}}(\widetilde{X}) \to D^b_{\operatorname{coh}}(X)$ being a Verdier localization means the category $D^b_{\operatorname{coh}}(X)$ can be expressed as a sort of quotient category of $D^b_{\operatorname{coh}}(\widetilde{X})$ with the induced triangulated structure. Conjecture~\ref{conj:bondal_orlov_localization} guarantees examinations on $D^b_{\operatorname{coh}}(\widetilde{X})$
descend to $D^b_{\operatorname{coh}}(X)$, and it has been recently confirmed for quotient singularities in characteristic zero \cite{Pavic/Shinder:2021}, cones over a projectively normal smooth Fano variety \cite{Efimov:2020}, and singularities admitting a resolution with one-dimensional fibers \cite{Bondal/Kap/Schechtman:2018}.

There is another methodology for extracting information from $D^b_{\operatorname{coh}}(X)$, and this is done by working object-wise. The following gadgets for $D^b_{\operatorname{coh}}(X)$ were introduced in \cite{BVdB:2003, ABIM:2010, Rouquier:2008}, and they break up the homological information into atomic-like bits. Roughly speaking, an object $G$ in $D^b_{\operatorname{coh}}(X)$ is said to be a \textit{classical generator} if every object $E$ in $D^b_{\operatorname{coh}}(X)$ can be finitely built from $G$ by taking cones, shifts, and direct summands. The minimal number of cones needed to finitely build $E$ from $G$ is called the \textit{level} of $E$ with respect to $G$ and is denoted $\operatorname{level}_{D^b_{\operatorname{coh}}(X)}^G (E)$. If the maximum of $\operatorname{level}_{D^b_{\operatorname{coh}}(X)}^G (E)$ is finite as one ranges over all objects $E$ in $D^b_{\operatorname{coh}}(X)$, then $G$ is called a \textit{strong generator}, and this value is called its \textit{generation time}. The \textit{Rouquier dimension} of $D^b_{\operatorname{coh}}(X)$ is the minimal generation time amongst all strong generators and is denoted $\dim D^b_{\operatorname{coh}}(X)$. 

There are many naturally occurring instances where classical generators exist in $D^b_{\operatorname{coh}}(X)$. Recall a Noetherian scheme $X$ is said to be \textit{$J\textrm{-}2$} if for all morphisms $Y \to X$ locally of finite type, the regular locus of $Y$ is open. A few familiar cases include varieties over a field and algebras of finite type over the integers. In \cite{Elagin/Lunts/Schnurer:2020} it was shown $D^b_{\operatorname{coh}}(X)$ admits a classical generator whenever $X$ is a Noetherian $J\textrm{-}2$ scheme. However, there are not many general results that allow one to identify such objects in $D^b_{\operatorname{coh}}(X)$, which is a motivating problem for our work.

Returning to strong generation, an important case where one can explicitly describe a strong generator in $D^b_{\operatorname{coh}}(X)$ is for smooth quasi-projective varieties over a field \cite{Rouquier:2008,Orlov:2009}. However, once the variety is not smooth, the story becomes far more complicated. It has been exhibited that strong generators capture inherent properties \cite{BILMP:2023,Pollitz:2019,Iyengar/Takahashi:2019,Iyengar/Takahashi:2016,ABIM:2010}. The following major problem poses a deep connection between two invariants which appear unrelated at first glance \cite{Orlov:2009}.

\begin{introconjecture}\label{conj:orlov_dimension}
    The Krull dimension of a smooth projective variety $X$ over a field is the Rouquier dimension of $D^b_{\operatorname{coh}}(X)$.
\end{introconjecture}

Conjecture~\ref{conj:orlov_dimension} has been confirmed for various types of varieties, including projective spaces, quadrics, Grassmannians \cite{Rouquier:2008}; Del Pezzo surfaces and Fano threefolds of type $V_5$ or $V_{22}$ \cite{Ballard/Favero:2012}; smooth proper curves \cite{Orlov:2009}; a product of two Fermat elliptic curves, a Fermat $K3$ surface \cite{Ballard/Favero/Katzarkov:2014}; and toric varieties \cite{Hanlon/Hicks/Lazarev:2023}. If $X$ is a $d$-dimensional variety with rational singularities, then validity of Conjecture~\ref{conj:bondal_orlov_localization} and Conjecture~\ref{conj:orlov_dimension} promises the Rouquier dimension of $D^b_{\operatorname{coh}}(X)$ is $d$.

Recently, \cite{Aoki:2021} has shown $D^b_{\operatorname{coh}}(X)$ admits a strong generator when $X$ is a Noetherian quasi-excellent separated scheme of finite Krull dimension- i.e varieties over a field. This enriches the work of \cite{BVdB:2003, Rouquier:2008, Iyengar/Takahashi:2016, KMVdB:2011, Lunts:2010, Neeman:2021}. The results in \cite{Aoki:2021} were established by working in the unbounded derived category of complexes with quasi-coherent cohomology which is denoted by $D_{\operatorname{Qcoh}}(X)$. An important property satisfied by the triangulated category $D_{\operatorname{Qcoh}}(X)$ is that all small coproducts exist, and \cite{Aoki:2021} exploited this to detect strong generation in $D^b_{\operatorname{coh}}(X)$ as follows. Roughly speaking, an object $G$ in $D_{\operatorname{Qcoh}}(X)$ is a \textit{strong $\bigoplus$-generator} when every object of $D_{\operatorname{Qcoh}}(X)$ can be finitely built from $G$ using direct summands of small coproducts, shifts, and cones in $D_{\operatorname{Qcoh}}(X)$ where the number of cones required for any object has a uniform upper bound. This was introduced in \cite{BVdB:2003,Rouquier:2008}. It was shown in \cite{Aoki:2021} that there exists an object in $D^b_{\operatorname{coh}}(X)$ which is a strong $\oplus$-generator for $D_{\operatorname{Qcoh}}(X)$, and any such object is a strong generator for $D^b_{\operatorname{coh}}(X)$.

In light of these remarks, this work studies the behavior of generation along the derived pushforward induced by a morphism of Noetherian schemes. Our choice to study this behavior under the derived pushforward as opposed to something such as derived pullback is that boundedness for cohomology is generally not preserved by the latter. However, if the morphism of schemes is proper, then the derived pushforward will preserve bounded complexes with coherent cohomology \cite{Grothendieck:1960}. This property which is enjoyed by a proper morphism will be vital for studying strong generation in $D^b_{\operatorname{coh}}(X)$.

We start by first establishing a condition for when a tensor-exact functor between tensor triangulated categories that are rigidly-compactly generated preserves strong $\oplus$-generators. The application of interest for geometry is the derived pushforward of strong $\bigoplus$-generators remain such along a morphism $\pi \colon Y \to X$ of Noetherian schemes, cf. Remark~\ref{rmk:descent_conditions_for_schemes_big_generation}. The result below reappears later as Theorem~\ref{thm:descent_conditions}.

\begin{introthm}\label{introthm:descent_conditions}
    Suppse $\pi^\ast \colon \mathcal{T} \to \mathcal{S}$ is a tensor-exact functor between tensor triangulated categories that are rigidly-compactly generated. Assume $\pi^\ast$ preserves all small coproducts. Let $\pi_\ast$ be the right adjoint of $\pi^\ast$. If $G$ is a strong $\bigoplus$-generator for $\mathcal{T}$, the following are equivalent:
    \begin{itemize}
        \item $\pi_\ast G$ is a strong $\bigoplus$-generator in $\mathcal{S}$;
        \item $1_\mathcal{S}$ is an object of $\overline{\langle \pi_\ast G \rangle}_n$ in $\mathcal{S}$ for some $n\geq 0$.
    \end{itemize}
    In particular, if $\overline{\langle G \rangle}_l = \mathcal{T}$ and $1_\mathcal{S} \in \overline{\langle \pi_\ast G \rangle}_n$, then $\overline{\langle \pi_\ast G \rangle}_{ln} = \mathcal{S}$.
\end{introthm}

The next result provides a method to probe the finiteness of Rouquier dimension for $D^b_{\operatorname{coh}}(X)$ by looking at an open affine cover of $X$. This appears later as Theorem~\ref{thm:finite_rouquier_iff_locally}.

\begin{introthm}\label{introthm:finite_rouquier_iff_locally}
    If $X$ is a Noetherian separated scheme, then the following are equivalent:
    \begin{enumerate}
        \item $D^b_{\operatorname{coh}}(X)$ has finite Rouquier dimension;
        \item $D^b_{\operatorname{coh}}(U)$ has finite Rouquier dimension for all open affine subschemes $U$ in $X$. 
    \end{enumerate}
\end{introthm}

An immediate consequence to Theorem~\ref{introthm:descent_conditions} is that in many familiar instances, classical generation is preserved along the derived pushforward of a proper surjective morphism. This mild hypothesis yields the following upgrade, and it is stated later in Theorem~\ref{thm:proper_surjective_descent}.

\begin{introthm}\label{introthm:proper_surjective_descent}
    Suppose $\pi \colon Y\to X$ is a proper surjective morphism where $X$ is a Noetherian $J\textrm{-}2$ scheme of finite Krull dimension. If $G$ is a classical generator for $D^b_{\operatorname{coh}}(Y)$, then $\mathbb{R}\pi_\ast G$ is a classical generator for $D^b_{\operatorname{coh}}(X)$.
\end{introthm}

If both $D^b_{\operatorname{coh}}(Y)$ and $D^b_{\operatorname{coh}}(X)$ have finite Rouquier dimension, then Theorem~\ref{introthm:proper_surjective_descent} promises the preservation of strong generation. Moreover, Theorem~\ref{introthm:proper_surjective_descent} produces new examples of how to explicitly describe strong generators in practice. For instance, Example~\ref{ex:resolution_char_zero} exhibits strong generation is preserved by the derived pushforward of a resolution of singularities $\pi \colon \widetilde{X}\to X$ for a variety $X$ over a field of characteristic zero. More concretely, if $X$ is a $d$-dimensional complex quasi-projective variety, then there exists a projective resolution of singularities $\pi \colon \widetilde{X}\to X$ \cite{Hironaka:1964a}. If $\mathcal{L}$ is a very ample line bundle on $\widetilde{X}$, then $\mathcal{O}_{\widetilde{X}}\oplus \mathcal{L} \oplus \cdots \oplus \mathcal{L}^{\otimes d}$ is a strong generator for $D^b_{\operatorname{coh}}(\widetilde{X})$ \cite{Orlov:2009}. By Theorem~\ref{introthm:proper_surjective_descent}, we have $\mathbb{R}\pi_\ast (\mathcal{O}_{\widetilde{X}}\oplus \mathcal{L} \oplus \cdots \oplus \mathcal{L}^{\otimes d})$ is a strong generator for $D^b_{\operatorname{coh}}(X)$.

Another utility of Theorem~\ref{introthm:proper_surjective_descent} is it yields new techniques to obtaining upper bounds on the Rouquier dimension of an affine scheme. For example, the Rouquier dimension of $D^b_{\operatorname{coh}}(X)$ for $X$ an affine variety over a field of characteristic zero can be bounded above by a multiple of the Krull dimension of $X$ and level of $\mathcal{O}_X$ with respect to a strong generator. For further details, see Remark~\ref{rmk:effective_bound_rouquier_dimension_affine_variety}. Additionally, when $X$ is a derived splinter, the Rouquier dimension of $D^b_{\operatorname{coh}}(X)$ can be expressed only in terms of Krull dimension of $X$ (see Example~\ref{ex:splinter_rouquier_bound}). The following result is Corollary~\ref{cor:rouquir_dimension_bounds_affine_variety}.

\begin{introcor}\label{introcor:rouquir_dimension_bounds_affine_variety}
    Suppose $X$ is an affine Noetherian scheme, and $\pi\colon \widetilde{X} \to X$ is a proper surjective morphism. If both $D^b_{\operatorname{coh}}(X)$ and $D^b_{\operatorname{coh}}(\widetilde{X})$ have finite Rouquier dimension, then 
    \begin{displaymath}
        \dim D^b_{\operatorname{coh}}(X)\leq (g  + 1) \cdot \operatorname{level} ^{\mathbb{R}\pi_\ast G} (\mathcal{O}_X)
    \end{displaymath}
    for any strong generator $G\in D^b_{\operatorname{coh}}(\widetilde{X})$ with generation time $g$.
\end{introcor}

\begin{ack}
    The author would like to thank Matthew R. Ballard, Andres J. Fernandez Herrero, Alapan Mukhopadhyay, and the referee for their valuable comments shared that led to improvements on the earlier versions of this work.
\end{ack}

\begin{addendum}
    The author would like to add comments regarding two important generalizations made since the announcement of this work. The recent work of \cite{Dey/Lank:2024} has vastly generalized Theorem~\ref{thm:proper_surjective_descent} to a similar statement for a proper surjective morphism of Noetherian schemes. Moreover, recent development regarding Rouquier dimension utilizing approximation by perfect complexes in \cite{Lank/Olander:2024} has produced further generalizations of many results in Section~\ref{sec:bounds_on_rouquier_dimension}.
\end{addendum}

\begin{notation}
    Let $X$ be a Noetherian scheme. We will be primarily interested in the following triangulated categories:
    \begin{enumerate}
        \item $D_{\operatorname{Qcoh}}(X)$ is the derived category of complexes of abelian sheaves on $X$
        \item $D_{\operatorname{Qcoh}}(X)$ is the triangulated subcategory of $D_{\operatorname{Qcoh}}(X)$ consisting of complexes that have quasi-coherent cohomology
        \item $D^b_{\operatorname{coh}}(X)$ is the triangulated subcategory consisting of bounded complexes with coherent cohomology.
    \end{enumerate}
\end{notation}

\section{Generation}
\label{sec:generation}

This section recalls notions of generation in triangulated categories, and the
primary interests are triangulated categories that can be constructed from
quasi-coherent sheaves on a Noetherian scheme. A few standard references
regarding triangulated categories and generation are respectively
\cite{Krause:2022,Huybrechts:2006} and
\cite{BVdB:2003,Rouquier:2008,Neeman:2021}. Let $\mathcal{T}$ be a triangulated
category with shift functor $[1]\colon \mathcal{T}\to\mathcal{T}$. To start, we remind ourselves of the classes of schemes that is of interest to our work.

\begin{definition}
    Suppose $R$ is a Noetherian ring, and denote the regular locus of $R$ by $\operatorname{reg}(R)$. We say $R$ is  
    \begin{enumerate}
        \item \textbf{$J\textrm{-}0$} when $\operatorname{reg}(R)$ contains a nonempty open subset of $\operatorname{Spec}(R)$;
        \item \textbf{$J\textrm{-}1$} when $\operatorname{reg}(R)$ is open in $\operatorname{Spec}(R)$;
        \item \textbf{$J\textrm{-}2$} when $S$ is $J\textrm{-}1$ for every $R$-algebra $S$ that is of finite type.
    \end{enumerate}
    If $X$ is a Noetherian scheme and for some $t=0,1,2$ there exists an open affine cover of $X$ by $J\textrm{-}t$ rings, then $X$ is said to be \textbf{$J\textrm{-}t$}. It can be verified that $X$ is $J\textrm{-}2$ if any scheme locally of finite type over $X$ is $J\textrm{-}1$ as well. See \cite[$\S 32$]{Matsumura:1989} for further details.
\end{definition}

\begin{example}
    The following highlight interesting examples of the above notions. All rings below are assumed to be Noetherian.
    \begin{enumerate}
        \item A non-reduced Artinian local ring has empty regular locus, so it is not $J\textrm{-}0$, but such rings are $J\textrm{-}1$. 
        \item There is an example of a $1$-dimensional integral domain that is not $J\textrm{-}0$ \cite[Example 1]{Hochster:1973}.
        \item Any local isolated singularity of nonzero Krull dimension or regular ring is $J\textrm{-}1$.
        \item Every $J\textrm{-}1$ integral domain is $J\textrm{-}0$.
        \item There exists a $3$-dimensional local integral domain that is $J\textrm{-}0$ but not $J\textrm{-}1$ \cite[Proposition 3.5]{Ferrand/Raynaud:1970};
        \item Any $J\textrm{-}2$ ring is $J\textrm{-}1$. This includes quasiexcellent, Artinian, $1$-dimensional local rings, and $1$-dimensional Nagata rings.
    \end{enumerate}
\end{example}

\begin{definition}
    If $R$ is a Noetherian ring, then we say $R$ is
    \begin{enumerate}
        \item \textbf{quasiexcellent} when it is both $J\textrm{-}2$ and a $G$-ring;
        \item \textbf{excellent} when it is quasiexcellent and universally catenary.
    \end{enumerate} 
    A Noetherian scheme is said to be \textbf{excellent} (resp. \textbf{quasiexcellent}) if it admits an open affine cover of such rings.
\end{definition}

\begin{example}
    The following are examples of quasiexcellent rings. Furthermore, any essentially of finite type algebra over these are also quasiexcellent. For instance, Dedekind domains with fraction field of characteristic zero or Noetherian complete local rings. Any Noetherian scheme which is of finite type over such rings are quasiexcellent, and note that this includes varieties over a field.
\end{example}

\begin{definition}(\cite{BVdB:2003,Rouquier:2008})
    Let $\mathcal{S}$ be a subcategory of $\mathcal{T}$.
    \begin{enumerate}
        \item A full triangulated subcategory of $\mathcal{T}$ is \textbf{thick} when it is closed under retracts. The smallest thick subcategory in $\mathcal{T}$ containing $\mathcal{S}$ is denoted $\langle \mathcal{S} \rangle$.
        \item Let $\langle G \rangle_0$ be the full subcategory consisting of all objects that are isomorphic to the zero object, and set $\langle G \rangle_1$ to be the full subcategory containing $G$ closed under shifts and retracts of finite coproducts. If $n\geq 2$, then $\langle G \rangle_n$ denotes the full subcategory of objects which are retracts of an object $E$ appearing in a distinguished triangle
        \begin{displaymath}
            A \to E \to B \to A[1]
        \end{displaymath}
        where $A\in \langle G \rangle_{n-1}$ and $B\in
        \langle G \rangle_1$. 
    \end{enumerate}
\end{definition}

\begin{remark}
    There exists an exhaustive filtration:
    \begin{displaymath}
        \langle G \rangle_0 \subseteq \langle G \rangle_1 \subseteq \cdots \subseteq \bigcup_{n=0}^\infty \langle G \rangle_n = \langle G \rangle.
    \end{displaymath}
\end{remark}

\begin{definition}(\cite{Rouquier:2008,BVdB:2003})
    \begin{enumerate}
        \item If $E\in \langle G \rangle$, then we say that $G$ \textbf{finitely builds} $E$.
        \item If there exists $G\in \mathcal{T}$ such that $\langle G \rangle = \mathcal{T}$, then $G$ is called a \textbf{classical generator}. Additionally, when there exists an $n\geq 0$ such that $\langle G \rangle_n = \mathcal{T}$, then $G$ is called a \textbf{strong generator}.
    \end{enumerate}
\end{definition}

\begin{example}
    Let $X$ be a Noetherian scheme. An object $P$ in $D^b_{\operatorname{coh}}(X)$ is said to be \textbf{perfect} if $P$ is there exists an open affine cover of $X$ such that on each component $P$ quasi-isomorphic to a bounded complex of finite free modules. The collection of perfect complexes in $D^b_{\operatorname{coh}}(X)$ is denoted $\operatorname{perf}X$. By \cite[Corollary 3.1.2]{BVdB:2003}, there exists $P\in \operatorname{perf}X$ such that $\operatorname{perf}X =\langle P \rangle$, and any such object is called a \textbf{compact generator}.
\end{example}

\begin{example}\label{ex:j2_scheme_classical_generator}
    If $X$ is a Noetherian $J\textrm{-}2$ scheme, then $D^b_{\operatorname{coh}}(X)$ admits a classical generator. This is \cite[Theorem 4.15]{Elagin/Lunts/Schnurer:2020}. 
\end{example}

\begin{example}\label{ex:quasi_excellent_strong_generator}
    If $X$ is a Noetherian quasiexcellent separated scheme of finite Krull
    dimension, then there exists a strong generator in $D^b_{\operatorname{coh}}(X)$, cf. \cite[Main Theorem]{Aoki:2021}. This class of schemes includes varieties over a field \cite[Theorem 7.38]{Rouquier:2008}.
\end{example}

\begin{example}
    Suppose $X$ is a Noetherian scheme of prime characteristic $p$. Recall the Frobenius morphism $F\colon X \to X$ is defined on sections by $r \mapsto r^p$. If this morphism is finite, then $D^b_{\operatorname{coh}}(X)$ admits a classical generator \cite[Corollary 3.9]{BILMP:2023}. Additionally, if $X$ is separated, then $D^b_{\operatorname{coh}}(X)$ admits a strong generator. For instance, if $X$ is a quasi-projective variety over a perfect field of prime characteristic and $\mathcal{L}$ is a very ample line bundle on $X$, then $F_\ast^e (\bigoplus^{\dim X}_{i=0} \mathcal{L}^{\otimes i})$ is a strong generator for $D^b_{\operatorname{coh}}(X)$.
\end{example}

\begin{definition}(\cite{Rouquier:2008, ABIM:2010})
    Let $E,G,K\in\mathcal{T}$. If $E\in \langle G \rangle$, then the \textbf{level} of $E$ with respect to $G$ is defined to be the minimal $n$ such that $E\in \langle G \rangle_n$, and it is denoted by $\operatorname{level} ^G(E)$. This value measures the minimal number of required cones to finitely build $E$ from $G$. If $K\in \mathcal{T}$ is a strong generator, then the \textbf{generation time} of $K$ is the minimal $n$ such that $\operatorname{level} ^K (E)\leq n+1$ for all $E\in \mathcal{T}$. The \textbf{Rouquier dimension} of $\mathcal{T}$ is the smallest integer $d$ such that $\langle G \rangle_{d+1} = \mathcal{T}$ for some $G\in \mathcal{T}$, which is denoted $\dim \mathcal{T}$.
\end{definition}

\begin{example}
    If $R$ is a Noetherian commutative ring and $E\in \operatorname{mod}R$, then 
    \begin{displaymath}
        \operatorname{level} ^R (E) = \operatorname{pdim}_R (E) +1
    \end{displaymath}
    where $\operatorname{pdim}_R$ denotes projective dimension of the $R$-module \cite[$\S 8$]{Christensen:1998}.
\end{example}

\begin{example}
    Let $R$ be an Artinian ring and denote by $\operatorname{Jac}(R)$ for its Jacobson radical. Then $R/\operatorname{Jac}(R)$ is a strong generator for  $D^b_{\operatorname{mod}} (R)$ with generation time at most the Loewy length of $R$ \cite[Proposition 7.37]{Rouquier:2008}.
\end{example}

\begin{definition}
    A map $f\colon J \to E$ is said to be \textbf{$G$-coghost} when the induced map on hom-sets
    \begin{displaymath}
        \operatorname{Hom}  (f,G) \colon \operatorname{Hom}  (E,G[n]) \to \operatorname{Hom}  (J,G[n])
    \end{displaymath}
    vanishes for all $n\in \mathbb{Z}$. An \textbf{$n$-fold $G$-coghost map} is the composition of $n$ $G$-coghost maps. The \textbf{coghost index} of $E$ with respect to $G$, denoted $\operatorname{cogin} ^G (E)$, is the minimal $n$ such that any $N$-fold $G$-coghost map $f\colon J \to E$ vanishes when $N\geq n$. There exists a notion of \textbf{ghost maps} and \textbf{ghost index} in $\mathcal{T}$, which is essentially the dual to this discussion, but it is not important for the results of this work.
\end{definition}
    
\begin{remark}
    In the special case where $\mathcal{T}=D^b_{\operatorname{mod}} (R)$ for a Noetherian commutative ring $R$,
    \begin{displaymath}
        \operatorname{level} ^G (E) = \operatorname{cogin} ^G (E).
    \end{displaymath}
    This is \cite[Theorem 24]{OS:2012}, and the reader is encouraged to see \cite{OS:2012, Letz:2021} for further details.
\end{remark}

\begin{definition}(\cite{Neeman:2021,Rouquier:2008})
    Assume $\mathcal{T}$ admits all small coproducts. Let $G$ be an object of $\mathcal{T}$. 
    \begin{enumerate}
        \item $\overline{\langle G \rangle}_0$ is the full subcategory consisting of all objects isomorphic to the zero object, and set $\overline{\langle G \rangle}_1$ be the smallest full subcategory containing $G$ closed under shifts and retracts of small coproducts. 
        \item If $n\geq 2$, then $\overline{\langle G \rangle}_n$ denotes the full subcategory of objects which are retracts of an object $E$ appearing in a distinguished triangle
        \begin{displaymath}
            A \to E \to B \to A[1]
        \end{displaymath}
        where $A\in \overline{\langle G \rangle}_{n-1}$ and $B\in
        \overline{\langle G \rangle}_1$.
        \item An object $G$ of $\mathcal{T}$ is a \textbf{strong $\bigoplus$-generator} if there exists an $n\geq 0$ such that $\overline{\langle G \rangle}_n = \mathcal{T}$.
    \end{enumerate}
\end{definition}

\begin{example}\label{ex:quasi_excellent_strong_oplus}
    If $X$ is a Noetherian quasiexcellent separated scheme of finite Krull dimension, then there exists $G\in D^b_{\operatorname{coh}}(X)$ which is a strong $\bigoplus$-generator for $D_{\operatorname{Qcoh}}(X)$ \cite[Main Theorem]{Aoki:2021}.
\end{example}

\begin{remark}\label{rmk:strong_oplus_to_strong_gerator_time_bound}
    Let $X$ be a Noetherian scheme. For any $G\in D^b_{\operatorname{coh}}(X)$ and integer $n\geq 0$, the following can be deduced from the proofs of Lemma 2.4-2.6 in \cite{Neeman:2021}:
    \begin{displaymath}
        \overline{ \langle G \rangle}_n \cap D^b_{\operatorname{coh}}(X) = \langle G \rangle_n.
    \end{displaymath}
    Consequently, if $\overline{ \langle G \rangle}_n = D_{\operatorname{Qcoh}}(X)$, then $D^b_{\operatorname{coh}}(X) = \langle G \rangle_n$.
\end{remark}

\begin{definition}(\cite{Verdier:1996})
    A sequence of triangulated categories 
    \begin{displaymath}
        \mathcal{S} \xrightarrow{i} \mathcal{T} \xrightarrow{q} \mathcal{V}
    \end{displaymath}   
    is called a \textbf{Verdier localization sequence} when $q \circ i$ is zero, $i$ is fully faithful, and the induced functor from the Verdier quotient $\mathcal{T} / \operatorname{im}(i)$ to $\mathcal{V}$ is an equivalence where $\operatorname{im}(i)$ is the essential image of $i$. Under these conditions, we say that $q$ is a \textbf{Verdier localization (functor)} and $\mathcal{V}$ is a \textbf{Verdier localization} of $\mathcal{T}$.
\end{definition}

\begin{remark}
    Let $X$ be a Noetherian scheme and $i \colon Z \to X$ a closed immersion.
    \begin{enumerate}
        \item If $H \in D_{\operatorname{Qcoh}}(X)$, then the \textbf{support} of $H$ is given by the subset
        \begin{displaymath}
            \operatorname{supp} (H) \colon = \lbrace p \in X \mid \bigoplus_{n \in 
        \mathbb{Z}}	\mathcal H^n(H)_p \neq 0 \rbrace.
        \end{displaymath}
        \item We say an object $H$ in $D_{\operatorname{Qcoh}}(X)$ is \textbf{supported on} $Z$ whenever $\operatorname{supp} (H) \subseteq Z$. If  $H$ is an object of $D^b_{\operatorname{coh}}(X)$, then $\operatorname{supp}(H) \subseteq X$ is closed, and $\operatorname{supp} (H)$ is viewed as a closed subscheme in $X$ via the reduced induced subscheme structure.
        \item The full subcategory consisting of objects in $D_{\operatorname{Qcoh}}(X)$ (resp. $D^b_{\operatorname{coh}}(X)$) that are supported on $Z$ is denoted $D_{\operatorname{Qcoh},Z} (X)$ (resp. $D^b_{\operatorname{coh},Z}(X)$).
    \end{enumerate}
\end{remark}

\begin{definition}
    Suppose $i\colon Z\to X$ is a closed immersion. If $E\in D^b_{\operatorname{coh}}(X)$, then we say $E$ is \textbf{scheme-theoretically supported} on $Z$ when there exists $E^\prime \in D^b_{\operatorname{coh}}(Z)$ such that $i_\ast E^\prime$ is quasi-isomorphic to $E$ in $D^b_{\operatorname{coh}}(X)$.
\end{definition}

\begin{remark}
	Suppose $Z$ is a closed subscheme of Noetherian $X$ with corresponding ideal
	sheaf $\mathcal{I}$. For $n\geq 1$, let $Z_n$ be the closed
	subscheme of $X$ corresponding to the ideal sheaf $\mathcal{I}^n$. If $C \in
	D^b_{\operatorname{coh}}( X)$ is supported on $Z$, then there is an $n\geq 1$ such that $C$ is scheme-theoretically supported on $Z_n$. This is \cite[Lemma 7.40]{Rouquier:2008}, and note that separatedness was initially assumed, but is not necessary.
\end{remark}

\begin{example}
    Let $X$ be a Noetherian scheme. If $j\colon U \to X$ denotes an open immersion, then there exists a Verdier localization sequence
    \begin{displaymath}
        D^b_{\operatorname{coh},Z}(X) \to D^b_{\operatorname{coh}}(X)\xrightarrow{\mathbb{L}j^\ast} D^b_{\operatorname{coh}}(U)
    \end{displaymath}
    where $Z= X \setminus U$. A reference for proof is \cite[Theorem 4.4]{Elagin/Lunts/Schnurer:2020}.
\end{example}

\begin{example}
    Let $X$ be a Noetherian scheme. Suppose $j\colon U \to X$ is an open immersion. Then there exists a Verdier localization sequence
    \begin{displaymath}
        D_{\operatorname{Qcoh},Z} (X) \to D_{\operatorname{Qcoh}}(X)\xrightarrow{\mathbb{L}j^\ast} D_{\operatorname{Qcoh}} (U)
    \end{displaymath}
    where $Z= X \setminus U$. This is an application of \cite[Lemma 3.4]{Rouquier:2010}
\end{example}

\section{Big generation descent}
\label{sec:big_descent}

This section studies the behavior of strong $\oplus$-generation by a tensor-exact functor between rigidly-compactly generated tensor triangulated categories, cf. Remark~\ref{rmk:tt_triangulated_categories} for background. The following result establishes a condition for when strong $\oplus$-generation is preserved.

\begin{theorem}\label{thm:descent_conditions}
    Suppse $\pi^\ast \colon \mathcal{T} \to \mathcal{S}$ is a tensor-exact functor between tensor triangulated categories that are rigidly-compactly generated. Assume $\pi^\ast$ preserves all small coproducts. Let $\pi_\ast$ be the right adjoint of $\pi^\ast$. If $G$ is a strong $\bigoplus$-generator for $\mathcal{T}$, the following are equivalent:
    \begin{itemize}
        \item $\pi_\ast G$ is a strong $\bigoplus$-generator in $\mathcal{S}$;
        \item $1_\mathcal{S}$ is an object of $\overline{\langle \pi_\ast G \rangle}_n$ in $\mathcal{S}$ for some $n\geq 0$.
    \end{itemize}
    In particular, if $\overline{\langle G \rangle}_l = \mathcal{T}$ and $1_\mathcal{S} \in \overline{\langle \pi_\ast G \rangle}_n$, then $\overline{\langle \pi_\ast G \rangle}_{ln} = \mathcal{S}$.
\end{theorem}

There is a special case where Theorem~\ref{thm:descent_conditions} can be applied in the context of schemes, cf. Remark~\ref{rmk:descent_conditions_for_schemes_big_generation}. The next result provides a method to probe the finiteness of Rouquier dimension for $D^b_{\operatorname{coh}}(X)$ by looking at an open affine cover of $X$.

\begin{theorem}\label{thm:finite_rouquier_iff_locally}
    If $X$ is a Noetherian separated scheme, then the following are equivalent:
    \begin{enumerate}
        \item $D^b_{\operatorname{coh}}(X)$ has finite Rouquier dimesion;
        \item $D^b_{\operatorname{coh}}(U)$ has finite Rouquier dimesion for all open affine subschemes $U$ in $X$.
    \end{enumerate}
\end{theorem}

\subsection{Proofs}
\label{sec:big_descent_results}

We start by first proving Theorem~\ref{thm:descent_conditions}. To do so, a few elementary results are needed. The following remarks will be freely used in the section.

\begin{remark}\label{rmk:tt_triangulated_categories}
    \begin{enumerate}
        \item A \textbf{tensor triangulated category} is a triangulated category with a compatible closed symmetric monoidal strucure, cf. \cite[Appendix A.2]{Hovey/Palmieri/Strickland:1997}. 
        \item If $\mathcal{T}$ is a tensor triangulated category, then we denote the tensor by $\otimes \colon \mathcal{T} \times \mathcal{T} \to \mathcal{T}$ and tensor unit by $1_{\mathcal{T}}$.
        \item A tensor triangulated category is said to be \textbf{rigidly-compactly generated} if it is compactly generated and compact objects coincides with those that are rigid. See \cite[Definition 2.7]{Balmer/Dell'Ambrogio/Sanders:2016}.
        \item A \textbf{tensor-exact} functor $F \colon \mathcal{T} \to \mathcal{S}$ between tensor triangulated categories is an exact functor that is strong symmetric monoidal.
    \end{enumerate} 
\end{remark}

\begin{remark}
    Let $\pi^\ast \colon \mathcal{T} \to \mathcal{S}$ be tensor-exact functor between rigidly-compactly generated tensor triangulated categories. If $\pi^\ast$ preserves small coproducts, then $\pi^\ast$ admits a right adjoint $\pi_\ast \colon \mathcal{S} \to \mathcal{T}$ and there is a projection formula $A \otimes \pi_\ast (B) \cong \pi_\ast ( B \otimes \pi^\ast A)$ for all $A\in \mathcal{S}$ and $B \in \mathcal{T}$. This is \cite[Corollary 2.14]{Balmer/Dell'Ambrogio/Sanders:2016}.
\end{remark}

\begin{lemma}\label{lem:tensor_big_thick}
    Let $\mathcal{T}$ be a tensor triangulated category which admits small coproducts. If $E,G$ are objects of $\mathcal{T}$ and $j\geq 0$, then $\overline{\langle G \rangle}_j  \otimes  E$ is contained in $\overline{\langle G  \otimes  E\rangle}_j$.
\end{lemma}

\begin{proof}
    This is shown by induction on $j$. Note there exists nothing to show when $j=0$, so we may suppose $j=1$. If $T\in \overline{\langle G \rangle}_1  \otimes  E$, then $T\cong T^\prime  \otimes  E$ for some $T^\prime\in \overline{\langle G \rangle}_1$. Since $T^\prime\in \overline{\langle G \rangle}_1$, we know that  $T^\prime$ is a direct summand of a small coproduct of shifts of $G$, say $T^\prime \oplus T^{\prime \prime} \cong \bigoplus_{i \in I} G[r_i]$ for an indexing set $I$. From the following computation, we see that $T^\prime  \otimes  E \in \overline{\langle G  \otimes  E \rangle}_1$:
    \begin{displaymath}
        \begin{aligned}
            (T^\prime  \otimes  E) \oplus (T^{\prime \prime} \otimes  E) &\cong (T^\prime \oplus T^{\prime \prime}) \otimes  E \\&\cong (\bigoplus_{i \in I} G[r_i]) \otimes  E
            \\&\cong \bigoplus_{i \in I} (G \otimes  E)[r_i].
        \end{aligned}
    \end{displaymath}
    This establishes the base case.

    Suppose there exists $n\geq 1$ such that the claim holds for $0\leq j \leq n$. Let $T\in \overline{\langle G \rangle}_{n+1}  \otimes  E$. There exists $T^\prime \in \overline{\langle G \rangle}_{n+1}$ such that $T\cong T^\prime  \otimes  E$. If $T^\prime \in \overline{\langle G \rangle}_{n+1}$, then there exists a distinguished triangle
    \begin{displaymath}
        A \to T^\prime\oplus T^{\prime \prime} \to B \to A[1]
    \end{displaymath}
    where $A\in \overline{\langle G \rangle}_n$ and $B\in \overline{\langle G \rangle}_1$. Applying the exact endofunctor $(-) \otimes  E$ on $\mathcal{T}$, this yields another distinguished triangle
    \begin{displaymath}
        A  \otimes  E \to (T^\prime\oplus T^{\prime \prime}) \otimes  E \to B \otimes  E \to (A  \otimes  E)[1].
    \end{displaymath}
    By the induction hypothesis, we know that $A \otimes  E \in \overline{ \langle G \otimes  E \rangle}_n$ and $B\in \overline{ \langle G \otimes  E \rangle}1$. Hence, it follows that
    \begin{displaymath}
        (T^\prime \otimes  E) \oplus (T^{\prime \prime}  \otimes  E) \cong (T^\prime\oplus T^{\prime \prime}) \otimes  E \in \overline{\langle G \otimes  E\rangle}_{n+1},
    \end{displaymath}
    which completes the proof.
\end{proof}

\begin{lemma}\label{lem:big_thick_compare}
    Let $\mathcal{T}$ be a triangulated category admitting small coproducts and $E,G$ objects in $\mathcal{T}$. If $E\in \overline{\langle G \rangle}_n$ and $j\geq 0$, then $\overline{\langle E \rangle}_j$ is contained in $\overline{\langle  G \rangle}_{jn}$.
\end{lemma}

\begin{proof}
    This will be shown by induction on $j$. There exists nothing to check when $j=0$, so we may suppose $j=1$. If $E\in \overline{\langle G \rangle}_n$, then there exists a distinguished triangle 
    \begin{displaymath}
        A \to E \oplus E^\prime \to B \to A[1]
    \end{displaymath}
    where $A\in \overline{\langle G \rangle}_{n-1}$ and $B\in \overline{\langle G \rangle}_1$. Any small coproduct of this distinguished triangle remains so, and hence, $\overline{ \langle E \rangle}_1 \subseteq \overline{\langle G \rangle}_n$. This establishes the base case. 
    
    Assume there exists $j\geq 1$ such that the claim holds for each $0\leq c \leq j$. Let $T\in \overline{\langle  E \rangle}_{j+1}$. There exists a distinguished triangle
    \begin{displaymath}
        A^\prime \to T\oplus T^\prime \to B^\prime \to A^\prime[1]
    \end{displaymath}
    where $A^\prime\in \overline{\langle  E \rangle}_n$ and $B^\prime\in \overline{ \langle E \rangle}_1$. The induction step tells us $A^\prime\in \overline{ \langle G \rangle}_{jn}$ and $B^\prime \in \overline{ \langle G \rangle}_j$. Therefore, $T\in \overline{\langle  E \rangle}_{(j+1)n}$ as desired, and this completes the proof.
\end{proof}

\begin{proof}[Proof of Theorem~\ref{thm:descent_conditions}]
    If $\pi_\ast G$ is a strong $\bigoplus$-generator for $\mathcal{S}$, then there exists $n\geq 0$ such that $\mathcal{S}=\overline{ \langle \pi_\ast G \rangle}_n$, and so $1_\mathcal{S} \in \overline{ \langle \pi_\ast G \rangle}_n$. Hence, we only need to check the backward direction.

    Suppose $G$ is a strong $\bigoplus$-generator for $\mathcal{T}$ such that $1_\mathcal{S} \in \overline{ \langle \pi_\ast G \rangle}_n$ for some $n \geq 0$. Given $E \in \mathcal{S}$ and $j\geq 0$, Lemma~\ref{lem:tensor_big_thick} ensures $\overline{ \langle \pi_\ast G \rangle}_j \otimes E$ is contained in $\overline{\langle  \pi_\ast G \otimes E  \rangle}_j$. The projection formula promises $\pi_\ast G \otimes E \cong  \pi_\ast (G \otimes \mathbb{L} \pi^\ast E)$, and so $\overline{ \langle \pi_\ast G \rangle}_j \otimes E$ is contained in $\overline{\langle   \pi_\ast (G \otimes \mathbb{L} \pi^\ast E) \rangle}_j$. If $\mathcal{O}_X \in \overline{\langle \pi_\ast G \rangle}_n$, then $E\in \overline{\langle  \pi_\ast (G \otimes \mathbb{L} \pi^\ast E)  \rangle}_n$. Choosing $l\geq 0$ such that $\mathcal{T} = \overline{\langle G \rangle}_l$, 
    we have $G\otimes \mathbb{L} \pi^\ast E \in \overline{\langle G \rangle}_l$, and so
    \begin{displaymath}
        \pi_\ast (G\otimes \mathbb{L} \pi^\ast E) \in \overline{ \langle \pi_\ast G \rangle}_l.
    \end{displaymath}
    Furthermore, via Lemma~\ref{lem:big_thick_compare}, $E$ is an object of $\overline{ \langle \pi_\ast (G \otimes \mathbb{L} \pi^\ast E) \rangle}_n$, and so it is contained in $\overline{\langle \pi_\ast G\rangle}_{ln}$. As $E$ is arbitrary, this shows $\pi_\ast G$ is a strong $\bigoplus$-generator in $\mathcal{S}$, which completes the proof.
\end{proof}

\begin{remark}\label{rmk:descent_conditions_for_schemes_big_generation}
    Suppose $\pi \colon Y \to X$ is a proper surjective morphism of Noetherian schemes. Let $G$ be a strong $\oplus$-generator for $D_{\operatorname{Qcoh}}(Y)$. By Theorem~\ref{thm:descent_conditions}, we see that $\mathbb{R}\pi_\ast G$ is a strong $\oplus$-generator for $D_{\operatorname{Qcoh}}(X)$ if, and only if, there exists an $n\geq 0$ such that $\mathcal{O}_X \in \overline{ \langle \mathbb{R}\pi_\ast G \rangle}_n$. For example, if $Y$ is a separated quasi-excellent Noetherian scheme of finite Krull dimension, then there exists $G$ in $D^b_{\operatorname{coh}}(Y)$ which is a strong $\oplus$-generator for $D_{\operatorname{Qcoh}}(Y)$ \cite[Main Theorem]{Aoki:2021}. Hence, if $\mathcal{O}_X \in \langle \mathbb{R}\pi_\ast G \rangle$, then $\mathbb{R}\pi_\ast G$ is a strong generator for $D^b_{\operatorname{coh}}(X)$ in light of Theorem~\ref{thm:descent_conditions} and Remark~\ref{rmk:strong_oplus_to_strong_gerator_time_bound}.
\end{remark}

After establishing Theorem~\ref{thm:descent_conditions}, we next embark on proving Theorem~\ref{thm:finite_rouquier_iff_locally}. The technology introduced in \cite{Neeman:2021} is a motivating theme in our strategy below.

\begin{lemma}\label{lem:pushforward_coprod}
    Let $\mathcal{S},\mathcal{T}$ be triangulated categories admitting small coproducts. Let $F\colon\mathcal{T} \to \mathcal{S}$ be an exact functor preserving small coproducts. If $n\geq 0$ and $G\in \mathcal{T}$, then $F(\overline{\langle G \rangle}_n)$ is contained in $\overline{\langle F(G)\rangle}_n$. 
\end{lemma}

\begin{proof}
    The claim follows by induction on $n$. If $n=0$, then there exists nothing to check, so we may suppose $n=1$. If $E\in \overline{\langle G \rangle}_1$, then $E$ is a direct summand of a coproduct of shifts of $G$. Hence, $F (E)$ is a direct summand of a coproduct of shifts of $F G$, and so $F (E) \in\overline{ \langle F (G) \rangle}_1$. This establishes the base case. 
    
    Assume there exists $n\geq 1$ such that for each $0\leq k \leq n$ the claim holds for any object contained in $\overline{\langle G \rangle}_k$. Let $E \in \overline{\langle G \rangle}_{n+1} $. There exists a distinguished triangle
    \begin{displaymath}
        A \to E \oplus E^\prime \to B \to A[1]
    \end{displaymath}
    where $A \in \overline{\langle G \rangle}_n$ and $B\in \overline{\langle G \rangle}_1$. Moreover, this yields another distinguished triangle in $\mathcal{S}$,
    \begin{displaymath}
        F (A) \to F (E) \oplus F (E^\prime) \to F (B) \to F (A[1]).
    \end{displaymath}
    From the inductive step, it follows that $F (A) \in \overline{\langle F (G )\rangle}_n$ and $F (B)\in \overline{\langle F (G )\rangle}_1$. Consequently, one has $F (E) \in \overline{ \langle F (G) \rangle}_{n+1}$, which completes the proof.
\end{proof}

\begin{lemma}\label{lem:sum_coprod}
    Let $\mathcal{T}$ be a triangulated category admitting small coproducts. If $H,G$ are objects of $\mathcal{T}$, then $\overline{\langle G \rangle}_n \oplus \overline{\langle H \rangle}_m$ is contained in $\overline{ \langle G \oplus H \rangle}_{\max\{n,m\}}$.
\end{lemma}

\begin{proof}
    Recall for every integer $k\geq 0$ and collection of objects $\mathcal{B}$ in $\mathcal{T}$, one has $\overline{\langle \mathcal{B}\rangle}_k \subseteq \overline{\langle \mathcal{B}\rangle}_{k+1}$. Since $n \leq m$ or $m \leq n$, we can reduce to the case $n=m$. The proof will follow by induction on $n$. Consider an object $E \in \overline{\langle G \rangle}_1 \oplus \overline{\langle H \rangle}_1$. Then $E$ is the direct sum of two objects $A\oplus B$ where $A \in \overline{\langle G \rangle}_1$ and $B \in \overline{\langle H \rangle}_1$. Both $A,B$ are respectively direct summands of coproducts of shifts of $G,H$. Hence, $E$ is a direct summand of coproduct of shifts of $G\oplus H$, which establishes the base case. 
    
    Assume now the claim holds whenever $k<n$, and let $E\in \overline{\langle G \rangle}_n \oplus \overline{\langle H \rangle}_n$. As above, $E$ is the direct sum of two objects $E_1\oplus E_2$ where $E_1 \in \overline{\langle G \rangle}_n$ and $E_2 \in \overline{\langle H \rangle}_n$. There exists $A_1 \in \overline{\langle G \rangle}_{n-1}$, $B_1 \in \overline{\langle G \rangle}_1$, $A_2 \in \overline{\langle H \rangle}_{n-1}$, and $B_2 \in \overline{\langle  H \rangle}_1$ yielding the following distinguished triangles:
    \begin{displaymath}
        \begin{aligned}
            & A_1 \to E_1 \to B_1 \to A_1 [1],
            \\& A_2 \to E_2 \to B_2 \to A_2 [1].
        \end{aligned}
    \end{displaymath}
    The inductive step ensures that $A_1 \oplus A_2 \in \overline{\langle G \oplus H \rangle}_n$ and $B_1 \oplus B_2 \in \overline{\langle  G \oplus H \rangle}_1$. There exists the distinguished triangle
    \begin{displaymath}
        A_1 \oplus A_2 \to E_1 \oplus E_2 \to B_1 \oplus B_2 \to (A_1\oplus A_2)[1],
    \end{displaymath}
    and hence, $E_1\oplus E_2 \in \overline{\langle  G \oplus H \rangle}_{n+1}$. This completes the proof by induction.
\end{proof}

\begin{remark}\label{rmk:perf_to_perf_big_generation}
    Let $X$ be a Noetherian separated scheme, and $i\colon U \to X$ an open immersion. Suppose $G\in \operatorname{perf}X$ is a compact generator for $D_{\operatorname{Qcoh}}(X)$. For any $P\in \operatorname{perf}U$, there exists $N\geq 0$ such that $\mathbb{R}i_\ast P \in \overline{\langle  G \rangle}_N$. This is a special case of \cite[Theorem 0.18 or Theorem 6.2]{Neeman:2021}.
\end{remark}

\begin{proof}[Proof of Theorem~\ref{thm:finite_rouquier_iff_locally}]
    If $s\colon U\to X$ is an open immersion, then the functor $\mathbb{L}s^\ast \colon D^b_{\operatorname{coh}}(X) \to D^b_{\operatorname{coh}}(U)$ is a Verdier localization. Hence, if $D^b_{\operatorname{coh}}(X)$ has finite Rouquier dimesion, then so does $D^b_{\operatorname{coh}}(U)$, which exhibits $(1)\implies (2)$.

    Next, we want to show the converse direction. This will be done by induction on the minimal number of components for an open affine cover of $X$. In particular, we will use the notion of generation in the derived category $D_{\operatorname{Qcoh}}(X)$ utilizing small coproducts along with a Mayer-Vietoris sequence to appeal to Remark~\ref{rmk:strong_oplus_to_strong_gerator_time_bound}. 
    
    If $X$ can be covered by only one open affine, then the desired claim holds, so assme there exists $N\geq 2$ and an open affine cover $X=\bigcup^N_{i=1}U_i$. Set $U=U_N$ and $V = \bigcup^{N-1}_{i=1} U_i$. Let $s\colon U \to X$ and $t \colon V \to X$ be the associated open immersions. The inductive hypothesis tells us there exists strong generators $G^\prime, G^{\prime \prime}$ respectively for $D^b_{\operatorname{coh}}(U), D^b_{\operatorname{coh}}(V)$. Recall that both $D^b_{\operatorname{coh}}(U)$ and $D^b_{\operatorname{coh}}(V)$ are Verdier localizations of $D^b_{\operatorname{coh}}(X)$. Hence, there exists objects $\widetilde{G^\prime},\widetilde{G^{\prime \prime}}$ in $D^b_{\operatorname{coh}}(X)$ such that $\mathbb{L}s^\ast \widetilde{G^\prime}\cong G^\prime$ in $D^b_{\operatorname{coh}}(U)$ and $\mathbb{L}t^\ast \widetilde{G^{\prime \prime}}\cong G^{\prime \prime}$ in $D^b_{\operatorname{coh}}(V)$. If we let $G:= \widetilde{G^\prime} \oplus \widetilde{G^{\prime \prime}}$, then $\mathbb{L}s^\ast G$ and $\mathbb{L}t^\ast G$ are respectively strong generators for $D^b_{\operatorname{coh}}(U)$ and $D^b_{\operatorname{coh}}(V)$.
    
    Suppose $P$ is a compact generator for $D_{\operatorname{Qcoh}}(X)$, cf. \cite[Theorem 3.1.1]{BVdB:2003} for existence thereof. By Remark~\ref{rmk:perf_to_perf_big_generation}, there exists $N_1,N_2$ such that $\mathbb{R}s_\ast \mathcal{O}_U \in \overline{\langle  P \rangle}_{N_1}$ and $\mathbb{R}t_\ast \mathcal{O}_V \in \overline{\langle  P \rangle}_{N_2}$. If $N=\max\{N_1,N_2\}$, then Lemma~\ref{lem:sum_coprod} promises $\mathbb{R}s_\ast \mathcal{O}_U \oplus \mathbb{R}t_\ast \mathcal{O}_V \in \overline{\langle  P \rangle}_N$.

    Set $W = U \cap V$, and let $j\colon W \to U$ be the associated open immersion. There exists an open immersion $l\colon W \to X$ such that $l=s \circ j$. For each $E\in D^b_{\operatorname{coh}}(X)$ there exists a Mayer-Vietoris type triangle (see \cite[Section 3.3]{BVdB:2003} or \cite[Section 5.2.4]{Rouquier:2008})
    \begin{displaymath}
        \mathbb{R} l_\ast \mathbb{L} l^\ast E [-1] \to E \to \mathbb{R} s_\ast \mathbb{L} s^\ast E \oplus \mathbb{R} t_\ast \mathbb{L} t^\ast E \to \mathbb{R} l_\ast \mathbb{L} l^\ast E.
    \end{displaymath}
    If $l= s \circ j$, then $\mathbb{R} l_\ast \mathbb{L} l^\ast = \mathbb{R} s_\ast \mathbb{R} j_\ast \mathbb{L} j^\ast \mathbb{L} s^\ast$, and so one can rewrite the triangle as
    \begin{displaymath}
        \begin{aligned}
            \mathbb{R} s_\ast & \mathbb{R} j_\ast \mathbb{L} j^\ast \mathbb{L} s^\ast E [-1] \to E \\&\to \mathbb{R} s_\ast \mathbb{L} s^\ast E \oplus \mathbb{R} t_\ast \mathbb{L} t^\ast E \to \mathbb{R} s_\ast \mathbb{R} j_\ast \mathbb{L} j^\ast \mathbb{L} s^\ast E.
        \end{aligned}
    \end{displaymath}
    Choose $a,b\geq 0$ such that $\langle \mathbb{L}s^\ast G \rangle_a = D^b_{\operatorname{coh}}(U)$ and $\langle \mathbb{L}t^\ast G \rangle_b = D^b_{\operatorname{coh}}(V)$. If $\mathbb{L}s^\ast E \in \langle \mathbb{L}s^\ast G \rangle_a$ and $\mathbb{L}t^\ast E \in \langle \mathbb{L}t^\ast G \rangle_b$, then $\mathbb{R}s_\ast \mathbb{L}s^\ast E \in \langle \mathbb{R}s_\ast \mathbb{L}s^\ast G \rangle_a$ and $\mathbb{R}t_\ast \mathbb{L}t^\ast E \in \langle \mathbb{R}t_\ast \mathbb{L}t^\ast G \rangle_b$ in $D_{\operatorname{Qcoh}}(X)$. However, as both $\mathbb{R}s_\ast \mathcal{O}_U $ and $\mathbb{R}t_\ast \mathcal{O}_V$ belong to $\overline{\langle  P \rangle}_N$, it follows from projection formula that $\mathbb{R}s_\ast \mathbb{L}s^\ast G$ and $\mathbb{R}t_\ast \mathbb{L}t^\ast G$ belong to $\overline{\langle  P \overset{\mathbb{L}}{\otimes} G \rangle}_N$. Set $c=\max\{a,b\}$. Then $\mathbb{R}s_\ast \mathbb{L}s^\ast E$ and $\mathbb{R}t_\ast \mathbb{L}t^\ast E$ belong to $\overline{\langle  P \overset{\mathbb{L}}{\otimes} G \rangle}_{cN}$.

    Since $U$ is affine, $\mathcal{O}_U$ is a compact generator for $D_{\operatorname{Qcoh}}(U)$. Once more, Remark~\ref{rmk:perf_to_perf_big_generation} tells us there exists $L \geq 0$ such that $\mathbb{R}j_\ast \mathcal{O}_W$ belongs to $\overline{\langle \mathcal{O}_U \rangle}_L$. If we tensor with $\mathbb{L}s^\ast E$, then we see that $\mathbb{R}j_\ast \mathbb{L}j^\ast \mathbb{L} s^\ast E$ belongs to $\overline{\langle \mathbb{L} s^\ast E \rangle}_L$. Hence, $\mathbb{R}j_\ast \mathbb{L}j^\ast \mathbb{L} s^\ast E$ belongs to $\overline{\langle \mathbb{L} s^\ast G \rangle}_{aL}$, and so $\mathbb{R}s_\ast \mathbb{R}j_\ast \mathbb{L}j^\ast \mathbb{L} s^\ast E$ belongs to $\overline{\langle \mathbb{R}s_\ast \mathbb{L} s^\ast G \rangle}_{aL}$. Consequently, it follows that $\mathbb{R}s_\ast \mathbb{R}j_\ast \mathbb{L}l^\ast \mathbb{L} s^\ast E$ belongs to $\overline{\langle P \overset{\mathbb{L}}{\otimes} G\rangle}_{a NL}$. We have shown that $E$ belongs to $\overline{\langle P \overset{\mathbb{L}}{\otimes} G \rangle}_{N (aL + c)}$. By Remark~\ref{rmk:strong_oplus_to_strong_gerator_time_bound}, it follows that $E\in \langle P\overset{\mathbb{L}}{\otimes} G \rangle_{N (aL + c)}$. These values $a,b,c,N,L$ are independent of $E$, so $P\overset{\mathbb{L}}{\otimes} G$ is a strong generator for $D^b_{\operatorname{coh}}(X)$ as desired.
\end{proof}

\section{Classical generation descent}
\label{sec:classical_descent}

This section establishes a condition for when classical generation is preserved by the derived pushforward functor of a proper surjective morphism of Noetherian schemes.

\begin{theorem}\label{thm:proper_surjective_descent}
    Suppose $\pi \colon Y\to S$ is a proper surjective morphism where $S$ is a Noetherian $J\textrm{-}2$ scheme of finite Krull dimension. If $G$ is a classical generator for $D^b_{\operatorname{coh}}(Y)$, then $\mathbb{R}\pi_\ast G$ is a classical generator for $D^b_{\operatorname{coh}}(S)$.
\end{theorem}

If both $D^b_{\operatorname{coh}}(Y)$ and $D^b_{\operatorname{coh}}(X)$
have finite Rouquier dimension, then
Theorem~\ref{thm:proper_surjective_descent} becomes a statement about the
preservation of strong generation. Before pressing forth with the proof of
Theorem~\ref{thm:proper_surjective_descent}, there are natural occurrences
where such results are applicable.

\begin{example}\label{ex:blowup_proper_surjective}
    Suppose $X$ is a Noetherian integral scheme. Let $Z\subseteq X$ be a closed subscheme, and $\pi \colon \widetilde{X}\to X$ the blowup of $X$ along $Z$. As $\pi$ is a proper morphism, its image $\pi(\widetilde{X})\subseteq X$ is closed. Furthermore, $\pi$ yields an isomorphism away from $Z$, i.e. $\pi^{-1}(X\setminus Z) \cong X\setminus Z$. Hence, $\pi(\widetilde{X})$ contains the generic point of $X$, so $\pi$ must be surjective.
\end{example}

\begin{example}\label{ex:resolution_char_zero}
    Consider a Noetherian quasiexcellent integral scheme $X$ of finite Krull dimension and of characteristic zero. By \cite[Theorem 2.3.6]{Temkin:2008}, there exists a resolution of singularities $\pi\colon \widetilde{X} \to X$. As $\pi$ is proper birational and $X$ is integral, it follows that $\pi$ is surjective. For instance, if $X$ is a quasi-projective variety over a field of characteristic zero and $\pi \colon \widetilde{X}\to X$ is a projective resolution of singularities, then $\bigoplus^{\dim X}_{i=0}\mathbb{R}\pi_\ast \mathcal{L}^{\otimes i}$ is a strong generator for $D^b_{\operatorname{coh}}(X)$ where $\mathcal{L}$ is any ample line bundle on $\widetilde{X}$. Note that $\bigoplus^{\dim X}_{i=0}\mathcal{L}^{\otimes i}$ is a strong generator for $D^b_{\operatorname{coh}}(\widetilde{X})$ \cite[Theorem 4]{Orlov:2009}. By \textit{projective} resolution of singularities $\pi \colon \widetilde{X}\to X$, we mean a projective birational morphism with $\widetilde{X}$ regular. If $X$ is quasi-projective, then there exists a projective resolution of singularities by \cite[Main Theorem 1]{Hironaka:1964a}.
\end{example}

\begin{remark}
   The assumption on the characteristic of the scheme or base field in Example~\ref{ex:resolution_char_zero} is not important, and is only made to ensure the existence of a resolution of singularities. If the resolution of singularities is known to exist for a particular variety, then the arguments carry over. More generally, Theorem~\ref{thm:proper_surjective_descent} can be applied to any proper dominant morphism between Noetherian integral $J\textrm{-}2$ schemes. For instance, a proper birational morphism between varieties over a field.
\end{remark}

\begin{example}
    Recall an exact functor $F\colon\mathcal{T} \to \mathcal{S}$ between triangulated categories is \textbf{essentially dense} if for each $S\in \mathcal{S}$ there exists $T\in \mathcal{T}$ such that $S$ is a retract of $F(T)$. Suppose $\pi \colon Y \to X$ is a proper birational morphism of projective varieties over a field such that $\mathbb{R}\pi_\ast \mathcal{O}_Y$ is isomorphic to $\mathcal{O}_X$ in $D_{\operatorname{Qcoh}}(X)$. This happens when $X$ has rational singularities and $Y$ is smooth. The functor $\mathbb{R}\pi_\ast \colon D^b_{\operatorname{coh}}(Y)\to D^b_{\operatorname{coh}}(X)$ is essentially dense via \cite[Lemma 7.4]{Kawamata:2022}.
\end{example}

\subsection{Proofs}
\label{sec:classical_descent_results}

To exhibit Theorem~\ref{thm:proper_surjective_descent}, we start by understanding how local behavior influences global behavior of objects finitely building one another.

\begin{lemma}\label{lem:generate_generically} 
    Suppose that $j\colon U \to X$ is an open subscheme, and let $E,P \in
    D^b_{\operatorname{coh}}( X)$. If $\mathbb{L}j^\ast E \in  \langle \mathbb{L}j^\ast P \rangle$ in $D^b_{\operatorname{coh}}( U)$, then there exists $C \in D^b_{\operatorname{coh},Z}( X)$ such that $E \in \langle P \oplus C \rangle$ in $ D^b_{\operatorname{coh}}( X)$ where $Z = X \setminus U$. 
\end{lemma}

\begin{proof}
    There exists an exact sequence of triangulated categories
    \begin{displaymath}
        0 \to D^b_{\operatorname{coh},Z}(X) \to  D^b_{\operatorname{coh}}( X)
        \xrightarrow{\mathbb{L} j^\ast}  D^b_{\operatorname{coh}}( U) \to 0, 
        \end{displaymath}
    and it allows one to view $D^b_{\operatorname{coh}}(U)$ as the Verdier
    quotient $D^b_{\operatorname{coh}}( X)/ D^b_{\operatorname{coh},Z}(X)$. From \cite[Proposition 2.3.1]{Verdier:1996}, the thick subcategories of
    $D^b_{\operatorname{coh}}(U)$ are in one-to-one correspondence with those in
    $D^b_{\operatorname{coh}}(X)$ containing $D^b_{\operatorname{coh},Z}(X)$. Let
    $\mathcal{T}$ denote $\langle P \oplus D^b_{\operatorname{coh},Z}(X)\rangle $ where $P \oplus D^b_{\operatorname{coh},Z}(X)$ is
    the full subcategory of $D^b_{\operatorname{coh}}(X)$ whose objects are the
    form $P \oplus H$ with $H\in D^b_{\operatorname{coh},Z}(X)$. Furthermore,
    denote by $\mathcal{T}_n$ for $\langle P \oplus
    D^b_{\operatorname{coh},Z}(X)\rangle_n$. By the construction of $\mathcal{T}$, one
    has $D^b_{\operatorname{coh},Z}(X) \subseteq \mathcal{T}$, and so its image
    under $\mathbb{L} j^\ast$ corresponds to a thick subcategory of $D^b
    (\operatorname{coh}U)$. Furthermore, the image of $\mathcal{T}$ under the
    exact functor $\mathbb{L} j^\ast: D^b_{\operatorname{coh}}(X)\to D^b
    (\operatorname{coh}U)$ lies inside of $\langle
    \mathbb{L} j^\ast P \rangle$. Observe that $\mathbb{L} j^\ast \mathcal{T} \subseteq
    D^b_{\operatorname{coh}}(U)$ is a thick subcategory and it contains
    $\mathbb{L}j^\ast P$, so there exists the reverse inclusion, i.e.
    $\mathbb{L}j^\ast \mathcal{T} = \langle \mathbb{L}
    j^\ast P \rangle$. To finish, this may be done by induction on
    $\operatorname{level}^\mathcal{T}  (-)$. If
    $\operatorname{level}^\mathcal{T}  (E)=1$, then $E$ is a direct
    summand of an object of the form $\bigoplus_{n \in \mathbb{Z}} (P^{\oplus
    r_n} \oplus H_n )[n]$ where $H_n \in D^b_{\operatorname{coh},Z}(X)$. Hence,
    $E$ is an element of $\langle P \oplus (\bigoplus_{n
    \in\mathbb{Z}} H_n) \rangle$, and since $\bigoplus_{n\in \mathbb{Z}} H_n \in D^b_{\operatorname{coh},Z}(X)$, this promises the base case holds. Suppose for any
    $H \in \mathcal{T}$ with $\operatorname{level}^\mathcal{T}  (H) < n$
    the claim holds. Choose $E\in \mathcal{T}_n$. There exists a distinguished
    triangle 
    \begin{displaymath}
        A \to \widetilde{E} \to B \to A[1]
    \end{displaymath}
    where $A \in \mathcal{T}_{n-1}, B \in \mathcal{T}_1$ and $E$ is a direct
    summand of $\widetilde{E}$. The induction step guarantees there exists
    $C_A,C_B \in D^b_{\operatorname{coh},Z}(X)$ such that $ A \in \langle  P \oplus C_A \rangle$ and $B \in \langle P \oplus C_B \rangle$. Therefore, tying this together it follows that $\widetilde{E} \in \langle P \oplus C_A \oplus C_B\rangle$, and hence, so does $E$. This completes the proof.
\end{proof}

\begin{proposition} \label{prop:structure_sheaves_generate} 
    Consider a Noetherian scheme $X$, and choose an integer $l\geq 0$. If $\mathcal{Z}$ denotes the full subcategory with objects $\mathbb{R}i_\ast \mathcal O_Z$ in $D^b_{\operatorname{coh}}(X)$ for all closed immersions $i\colon Z \to X$ with $Z$ a closed integral subscheme with $\dim Z \leq l$, then collection of objects $E$ in $D^b_{\operatorname{coh}}( X)$ such that $\dim \operatorname{supp} E \leq l$ belongs to the smallest thick subcategory generated by $\mathcal{Z}$.
\end{proposition}

\begin{proof}
    It is enough to show $\operatorname{coh}_{\leq l} X \subseteq
    \langle \mathcal{Z} \rangle$ where $\operatorname{coh}_{\leq
    l} X$ is the full subcategory of coherent sheaves of $E$ with $\dim
    \operatorname{supp} E \leq l$. Indeed, for any bounded complex $E\in D^b
    (\operatorname{coh} X)$, one has $E \in \langle
    \{\mathcal{H}^i(E) : i \in \mathbb{Z}\} \rangle$. Fix $E \in
    \operatorname{coh}_{\leq l} X$. Note that $\operatorname{supp}(E)$ is a
    closed subset of $X$, and as $X$ is Noetherian, $\operatorname{supp}(E)$ has
    finitely many irreducible components, say $\operatorname{supp}(E) =
    \bigcup^n_{i=1} Z_i$. Each $Z_j$ has dimension at most $l$. An induction argument on the number of irreducible components of the support allows one to reduce to the case $E$ is supported on a closed integral subscheme of $X$ of dimension at most $l$. We to show if one can finitely build all coherent sheaves whose support is contained in
    a closed integral subscheme of $X$ with dimension at most
    $\dim\operatorname{supp}(E)$, then one can finitely build $E$. Indeed, this
    can be shown by induction on $\dim\operatorname{supp}(E)$. If
    $\dim\operatorname{supp}(E) = 0$, then one can reduce to $E$ being
    pushed-forward from $\kappa(p)$ for some closed point $p\in X$. However, any
    such $E$ is a finite direct sum of shifts of $\kappa(p)$-vector spaces, and
    so it can be finitely built by $\kappa(p)$. Since $\kappa(p) \in
    \mathcal{Z}$, the base case follows.
    
    Assume all coherent sheaves in $D^b_{\operatorname{coh}}(X)$ whose support has dimension at most $l$ have been finitely built by $\mathcal{Z}$. Choose $E\in D^b_{\operatorname{coh}}(X)$ such that $\dim\operatorname{supp}(E)=l+1$ and is scheme theoretically supported
    on a closed integral subscheme $i\colon Z \to X$. There exists $E^\prime \in \operatorname{coh}Z$ such that $i_\ast E^\prime \cong E$ in $D^b_{\operatorname{coh}}(X)$. By generic freeness for coherent sheaves on a Noetherian integral scheme, we can find an open affine immersion $t\colon U\to Z$ and a quasi-isomorphism $j \colon \mathbb{L}t^\ast \mathcal{O}_Z^{\oplus r} \to \mathbb{L}t^\ast E^\prime$ in $D^b_{\operatorname{coh}}(U)$. Note that $U$ is nonempty as it contains the generic point of $Z$. Hence, there exists a complex $S$ and a roof in $D^b_{\operatorname{coh}}( Z)$,
    \begin{equation}\label{eq:thick_support_roof}
        \begin{tikzcd}
            & S \\
            {\mathcal{O}_Z^{\oplus r}} && {E^\prime}
            \arrow["g"', from=1-2, to=2-1]
            \arrow["h", from=1-2, to=2-3]
        \end{tikzcd}
    \end{equation}
    such that the cone of $g$ is supported on the closed subscheme $Z \setminus U$ in $Z$ and $j = \mathbb{L}t^\ast h \circ (\mathbb{L}t^\ast g)^{-1}$ in $D^b_{\operatorname{coh}}(U)$. Applying the functor $\mathbb{L}i^\ast$ to the roof in Diagram~\ref{eq:thick_support_roof}, the octahedral axiom for triangulated categories ensures there exists a commutative diagram in $D^b_{\operatorname{coh}}(U)$:
    \begin{equation}\label{eq:thick_support_octahedral}
        \begin{tikzcd}
            {\mathbb{L}t^\ast S} & {\mathbb{L}t^\ast \mathcal{O}_Z^{\oplus r}} & {\mathbb{L}t^\ast \operatorname{cone}(g)} \\
            & {\mathbb{L}t^\ast E^\prime} \\
            {\mathbb{L}t^\ast \operatorname{cone}(g)[1]} & 0 & {\mathbb{L}t^\ast \operatorname{cone}(h)}
            \arrow["{\mathbb{L}t^\ast g}", from=1-1, to=1-2]
            \arrow["{\mathbb{L}t^\ast h}"', from=1-1, to=2-2]
            \arrow["j"', from=1-2, to=2-2]
            \arrow[from=1-2, to=1-3]
            \arrow[from=2-2, to=3-2]
            \arrow[from=2-2, to=3-3]
            \arrow["\exists", from=1-3, to=3-3]
            \arrow["\exists", from=3-3, to=3-2]
            \arrow[from=3-2, to=3-1]
        \end{tikzcd}
    \end{equation}
    Further, Diagram~\ref{eq:thick_support_octahedral} yields a distinguished triangle 
    \begin{displaymath}
        \mathbb{L}t^\ast \operatorname{cone}(g) \to \mathbb{L}t^\ast \operatorname{cone}(h) \to 0 \to \mathbb{L}t^\ast \operatorname{cone}(g)[1].
    \end{displaymath}
    If $\operatorname{cone}(g)$ is supported on $Z\setminus U$, then $\mathbb{L}t^\ast \operatorname{cone}(g)\cong 0$ in $D^b_{\operatorname{coh}}(U)$, and so the distinguished triangle above promises that $\mathbb{L}t^\ast \operatorname{cone}(h)\cong 0$. Hence, $\operatorname{cone}(h)$ is supported on $Z\setminus U$. 
    
    Next up, we show that we can finitely build $E$ from $\mathcal{Z}$ in $D^b_{\operatorname{coh}}(X)$. There exists a distinguished triangle
    \begin{displaymath}
        S  \xrightarrow{g} \mathcal{O}_Z^{\oplus r}\to 
        \operatorname{cone}(g) \to S[1].
    \end{displaymath}
    Let $\mathcal{Z}^\prime$ denote the full subcategory of objects $p_\ast \mathcal{O}_S\in D^b_{\operatorname{coh}}(Z)$ for $S$ a closed integral subscheme of $Z$ with Krull dimension at most $l+1$. Since $\operatorname{cone}(g)$ is supported in dimension lower than $\operatorname{dim}(Z)$, the inductive step implies that $S$ is finitely built by $\mathcal{Z}^\prime$. However, $\operatorname{cone}(h)$ is supported dimension less than $\dim Z$ as well, so it must be finitely built by $\mathcal{Z}^\prime$. From the distinguished triangle
    \begin{displaymath}
        S  \xrightarrow{h} E^\prime \to 
        \operatorname{cone}(h) \to S[1],
    \end{displaymath}
    we have $E^\prime$ is finitely built by $S$ and $\operatorname{cone}(h)$, and hence, via $\mathcal{Z}^\prime$. After applying the functor $i_\ast \colon D^b_{\operatorname{coh}}(Z)\to D^b_{\operatorname{coh}}(X)$, we see that $E$ is finitely built by $i_\ast \mathcal{Z}^\prime$. But $i_\ast \mathcal{Z}^\prime$ is a subcategory of $\mathcal{Z}$, and so this completes the proof by induction.
\end{proof}

\begin{corollary}\label{cor:strucutre_sheaves_classically_generate}
    Let $X$ be a Noetherian scheme of finite Krull dimension. If $\mathcal{Z}$ denotes the collection of structures sheaves $\mathbb{R}i_\ast \mathcal{O}_Z$ where $i\colon Z\to X$ is a closed integral subscheme, then $D^b_{\operatorname{cohX}}(X) =\langle \mathcal{Z} \rangle$.
\end{corollary}

\begin{proof}
    If $X$ is has finite Krull dimension, then every $E\in D^b_{\operatorname{coh}}(X)$ is scheme-theoretically supported on a closed subscheme with finite Krull dimension, and so the claims follows immediately via Proposition~\ref{prop:structure_sheaves_generate}.
\end{proof}

\begin{proposition}\label{prop:classical_generator_components}
    Let $X$ be a Noetherian scheme of finite Krull dimension. If $X=\bigcup^n_{i=1} Z_i$ denotes the maximal irreducible components and $G_i \in D^b_{\operatorname{coh}}(Z_i)$ is a classical generator for each $1\leq i \leq n$, then $\oplus_{i=1}^n \mathbb{R} \pi_{i,\ast} G_i$ is a classical generator for $D^b_{\operatorname{coh}}(X)$ where $\pi_i \colon Z_i \to X$ is the closed immersion.
\end{proposition}

\begin{proof}
    The object $G$ finitely builds any structure sheaf $i_\ast \mathcal{O}_Z$ with $i \colon Z \to X$ a closed immersion from a closed integral subscheme. Hence, we conclude by appealing to Corollary~\ref{cor:strucutre_sheaves_classically_generate}.
\end{proof}

\begin{proof}[Proof of Theorem~\ref{thm:proper_surjective_descent}]
    Let us exhibit the claim by induction on the Krull dimension of $\dim S$. If $\dim S = 0$, then it suffices to show that $\mathbb{R}\pi_\ast G$ finitely builds each closed point in $S$. Fix a closed point $s\in S$. Let $i\colon \operatorname{Spec}(\kappa(s)) \to S$ denote the closed immersion. There exists a fibered square of schemes:
    \begin{equation}\label{eq:descent_proper_surjective_dim_zero}
        \begin{tikzcd}
            {Y\times_S \operatorname{Spec}(\kappa(s))} & {\operatorname{Spec}(\kappa(s))} \\
            Y & S.
            \arrow["{\pi_2}", from=1-1, to=1-2]
            \arrow["i", from=1-2, to=2-2]
            \arrow["{\pi_1}"', from=1-1, to=2-1]
            \arrow["\pi"', from=2-1, to=2-2]
        \end{tikzcd}
    \end{equation}
    Recall that a proper surjection morphism is stable under base change, so $\pi_2$ is proper surjective. Furthermore, closed immersions are stable under base change, so $\pi_1 \colon Y\times_S \operatorname{Spec}(\kappa(s)) \to Y$ is a closed immersion. If $Y$ is $J\textrm{-}2$ and Noetherian, then so is $Y\times_S \operatorname{Spec}(\kappa(s))$. Let $G^\prime$ be a classical generator for $D^b_{\operatorname{coh}}(Y\times_S \operatorname{Spec}(\kappa(s)))$. As $\kappa(s)$ is a field, we have $\mathbb{R}\pi_{2,\ast} G^\prime$ is isomorphic to a bounded complex of the form $\bigoplus_{n\in \mathbb{Z}} \mathcal{O}_{\operatorname{Spec}(\kappa(s))}^{\oplus r_n} [n]$ whose differential is zero. In particular, this promises that $\mathcal{O}_{\operatorname{Spec}(\kappa(s))}$ is a direct summand of $\mathbb{R}\pi_{2,\ast} G^\prime$ in $D^b_{\operatorname{coh}} (\operatorname{Spec}(\kappa(s)) )$. Applying the functor $i_\ast \colon D^b_{\operatorname{coh}}(\operatorname{Spec}(\kappa(s))) \to D^b_{\operatorname{coh}}(S)$, it follows that $i_\ast \mathcal{O}_{\operatorname{Spec}(\kappa(s))}$ is a direct summand of $i_\ast \mathbb{R}\pi_{2,\ast} G^\prime$ in $D^b_{\operatorname{coh}}(S)$. However, Equation~\ref{eq:descent_proper_surjective_dim_zero} ensures $i_\ast \mathbb{R}\pi_{2,\ast} G^\prime= \mathbb{R}\pi_\ast \pi_{1,\ast} G^\prime$ in $D^b_{\operatorname{coh}}(S) $. If $G$ finitely builds $\pi_{1,\ast} G^\prime$ in $D^b_{\operatorname{coh}}(Y)$, then $\mathbb{R}\pi_\ast G$ finitely builds $\mathbb{R}\pi_\ast \pi_{1,\ast} G^\prime$ in $D^b_{\operatorname{coh}}(S)$. Hence, $\mathbb{R}\pi_\ast G$ finitely builds $i_\ast \mathcal{O}_{\operatorname{Spec}(\kappa(s))}$ in $D^b_{\operatorname{coh}}(S)$. By Corollary~\ref{cor:strucutre_sheaves_classically_generate}, $\mathbb{R}\pi_\ast G$ is a classical generator for $D^b_{\operatorname{coh}}(S)$, which establishes the base case.

    Before proving the inductive step, let us reduce to the case where $S$ is an integral scheme. Let $i\colon Z \to S$ be the closed immersion associated to a maximal irreducible component of $S$. Consider the fibered square of schemes:
    \begin{equation}\label{eq:descent_proper_surjective_integral_step}
        \begin{tikzcd}
            {Y\times_S Z} & {Z} \\
            Y & S.
            \arrow["{\pi_2}", from=1-1, to=1-2]
            \arrow["i", from=1-2, to=2-2]
            \arrow["{\pi_1}"', from=1-1, to=2-1]
            \arrow["\pi"', from=2-1, to=2-2]
        \end{tikzcd}
    \end{equation}
    Recall closed immersions and proper surjective morphisms are stable under base change. Assume the claim holds for each maximal irreducible component of $S$. If $Y$ is $J\textrm{-}2$ and Noetherian, then so is $Y\times_S Z$. Let $G^\prime$ be a classical generator for $D^b_{\operatorname{coh}}(Y\times_S Z)$. The assumption promises $\mathbb{R}\pi_{2,\ast} G^\prime$ is a classical generator for $D^b_{\operatorname{coh}}(Z)$. From Equation~\ref{eq:descent_proper_surjective_dim_nonzero}, we know $i_\ast \mathbb{R}\pi_{2,\ast} G^\prime = \mathbb{R}\pi_\ast \pi_{2,\ast} G^\prime$, and this object is finitely built by $\mathbb{R}\pi_\ast G$ as $G$ is a classical generator for $D^b_{\operatorname{coh}}(Y)$. Since this holds for all maximal irreducible component of $X$, Proposition~\ref{prop:classical_generator_components} promises $\mathbb{R}\pi_\ast G$ is a classical generator for $D^b_{\operatorname{coh}}(S)$. Hence, this allows us to reduce to the case $S$ is integral.

    Next, we prove the inductive step. Suppose there exists an integer $d\geq 0$ such that the claim holds for any proper surjective morphism $\pi^\prime \colon Y^\prime \to S^\prime$ of Noetherian $J\textrm{-}2$ schemes of finite Krull dimension with $\dim S^\prime \leq d$. Let $\dim S = d+1$. Choose $E\in D^b_{\operatorname{coh}}(S)$. If $\dim \operatorname{Supp}(E) < d+1$, then $\mathbb{R}\pi_\ast G$ finitely builds $E$ in $D^b_{\operatorname{coh}}(S)$, and this can be seen by a similar argument to the base case. Indeed, there exists a closed immersion $i\colon W \to S$ such that $E$ is scheme-theoretically supported on $W$ and $W\subseteq \operatorname{Supp}(E)$. There exists a fibered square of schemes:
    \begin{equation}\label{eq:descent_proper_surjective_dim_nonzero}
        \begin{tikzcd}
            {Y\times_S W} & {W} \\
            Y & S.
            \arrow["{\pi_2}", from=1-1, to=1-2]
            \arrow["i", from=1-2, to=2-2]
            \arrow["{\pi_1}"', from=1-1, to=2-1]
            \arrow["\pi"', from=2-1, to=2-2]
        \end{tikzcd}
    \end{equation}
    Recall that a proper surjection morphism is stable under base change, so $\pi_2$ is proper surjective. Furthermore, closed immersions are stable under base change, so $\pi_1 \colon Y\times_S W \to Y$ is a closed immersion. If $Y$ is $J\textrm{-}2$ and Noetherian, then so is $Y\times_S W$. Let $G^\prime$ be a classical generator for $D^b_{\operatorname{coh}}(Y\times_S W)$. If $\dim W < d+1$, then the induction hypothesis promises that $\mathbb{R}\pi_{2,\ast} G^\prime$ is a classical generator for $D^b_{\operatorname{coh}}(W)$. Since $E$ is scheme-theoretically supported on $W$, there exists $E^\prime \in D^b_{\operatorname{coh}}(W)$ such that $i_\ast E^\prime\cong E$ in $D^b_{\operatorname{coh}}(S)$. However, $\mathbb{R}\pi_{2,\ast} G^\prime$ finitely builds $E^\prime$ in $D^b_{\operatorname{coh}}(W)$, and so $i_\ast \mathbb{R}\pi_{2,\ast} G^\prime$ finitely builds $i_\ast E^\prime$ in $D^b_{\operatorname{coh}}(S)$. From Equation~\ref{eq:descent_proper_surjective_dim_nonzero}, we know $i_\ast \mathbb{R}\pi_{2,\ast} G^\prime = \mathbb{R}\pi_\ast \pi_{2,\ast} G^\prime$, and this object is finitely built by $\mathbb{R}\pi_\ast G$ since $G$ is a classical generator for $D^b_{\operatorname{coh}}(Y)$. Hence, $E$ is finitely built by $\mathbb{R}\pi_\ast G$ in $D^b_{\operatorname{coh}}(S)$.

    Lastly, we need to check the case when $E$ has full support. Since $S$ is $J\textrm{-}2$, the regular locus of $S$ is open. As $Y$ is integral, the regular locus contains the generic point. Let $j\colon U \to S$ be an open immersion with $U$ nonempty, affine, and contained in the regular locus of $S$. Note that $\pi$ being surjective ensures this is a dominant morphism, and so $\mathbb{R}\pi_\ast G$ has full support. Therefore, the support of $\mathbb{L}j^\ast \mathbb{R}\pi_\ast G$ is all of $U$, and so \cite[Theorem 1.5]{Neeman:1992} ensures $\mathbb{L}j^\ast \mathbb{R}\pi_\ast G$ is a classical generator for $D^b_{\operatorname{coh}}(U)$. Hence, $\mathbb{L}j^\ast \mathbb{R}\pi_\ast G$ finitely builds $\mathbb{L}j^\ast E$ in $D^b_{\operatorname{coh}}(U)$. By Lemma~\ref{lem:generate_generically}, there exists $E^\prime \in D^b_{\operatorname{coh},Z} (S)$ such that $\mathbb{R}\pi_\ast G \oplus E^\prime$ finitely build $E$ in $D^b_{\operatorname{coh}}(S)$ where $Z=S \setminus U$. However, as $S$ is integral, $\dim Z < d+1$. This means $E^\prime$ is scheme-theoretically supported on a proper closed subset of $S$, and the work above promises that $\mathbb{R}\pi_\ast G$ finitely builds $E^\prime$. Therefore, $\mathbb{R}\pi_\ast G$ finitely builds $E$, which completes the proof by induction.
\end{proof}

\section{Upper bounds on Rouquier dimension}
\label{sec:bounds_on_rouquier_dimension}

Within this section, upper bounds are established on Rouquier dimension in situations where strong generation is preserved by the derived pushforward of a morphism of Noetherian schemes. Namely, the following statement is shown.

\begin{proposition}\label{prop:descend_bounds}
    Suppose $\pi \colon Y\to X$ is a proper surjective morphism of Noetherian $J\textrm{-}2$ schemes. Assume there exists $G\in D^b_{\operatorname{coh}}(Y)$ that is a strong $\bigoplus$-generator for $D_{\operatorname{Qcoh}}(Y)$. If for some $n\geq 0$ one has $D_{\operatorname{Qcoh}}(Y)=\overline{\langle G \rangle}_n$, then $\dim D^b_{\operatorname{coh}}(X)$ is at most $n\cdot \operatorname{level}^{\mathbb{R}\pi_\ast G}  (\mathcal{O}_X)$.
\end{proposition}

Consequently, Proposition~\ref{prop:descend_bounds} can be utilized to explicitly produce upper bounds that can be expressed in terms of two familiar invariants.

\begin{corollary}\label{cor:rouquir_dimension_bounds_affine_variety}
    Suppose $X$ is an affine Noetherian scheme, and $\pi\colon \widetilde{X} \to X$ is a proper surjective morphism. If both $D^b_{\operatorname{coh}}(X)$ and $D^b_{\operatorname{coh}}(\widetilde{X})$ have finite Rouquier dimension, then for any strong generator $G\in D^b_{\operatorname{coh}}(\widetilde{X})$ with generation time $g$:
    \begin{displaymath}
        \dim D^b_{\operatorname{coh}}(X)\leq (g  + 1) \cdot \operatorname{level} ^{\mathbb{R}\pi_\ast G} (\mathcal{O}_X).
    \end{displaymath}
\end{corollary}

Before continuing forward with the proofs of Proposition~\ref{prop:descend_bounds} and Corollary~\ref{cor:rouquir_dimension_bounds_affine_variety}, the following highlight interesting examples for which these upper bounds can be discussed.

\begin{example}\label{ex:rouquier_bound_curves}
    Let $\nu\colon \widetilde{X}\to X$ be the normalization of a Noetherian quasiexcellent separated integral scheme $X$ of Krull dimension one. The map $\pi$ is finite dominant, and so it is proper surjective. If $G\in D^b_{\operatorname{coh}}(\widetilde{X})$ is a strong generator, then $\mathbb{R}\pi_\ast G$ is a strong generator. For instance, if $X=\operatorname{Spec}(R)$, then \cite[Corollary 8.4]{Christensen:1998} ensures $\overline{\langle \mathcal{O}_{X_\nu} \rangle}_3 = D_{\operatorname{Qcoh}}(\mathcal{O}_{X_\nu})$, and an application of Proposition~\ref{prop:descend_bounds} implies $\dim D^b_{\operatorname{mod}} (R)$ is at most $3 \cdot \operatorname{level} ^{\nu_\ast R} (R)$. For instance, suppose $X$ projective curve over a field. If $X$ is smooth, then \cite[Theorem 6]{Orlov:2009} showed the Rouquier dimension of $D^b_{\operatorname{coh}}(X)$ is one. If $X$ is rational and
    has either nodal or cuspidal singularities, then \cite[Theorem
    10]{Burban/Drozd:2011} ensures this value is at most two. However, Proposition~\ref{prop:descend_bounds} yields bounds more generally outside of these cases.
\end{example}

\begin{remark}\label{rmk:effective_bound_rouquier_dimension_affine_variety}
    Utilizing Corollary~\ref{cor:rouquir_dimension_bounds_affine_variety}, there are bounds that can be made explicit in practice. For instance, suppose $X$ is an affine variety over a field of characteristic zero and $\pi\colon \widetilde{X}\to X$ is a resolution of singularities,
    \begin{displaymath}
        \dim D^b_{\operatorname{coh}}(X)\leq (2\dim X + 1) \cdot \operatorname{level} ^{\mathbb{R}\pi_\ast G} (\mathcal{O}_X) 
    \end{displaymath}
    for any strong generator $G\in D^b_{\operatorname{coh}}(\widetilde{X})$ with minimal generation time. Indeed, there exists an upper bound on $\dim D^b_{\operatorname{coh}}( \widetilde{X})$ by $2\dim X$, cf. \cite[Proposition 7.9]{Rouquier:2008}. If \cite[Conjecture 10]{Orlov:2009} holds, then there is a sharper bound,
    \begin{displaymath}
        \dim D^b_{\operatorname{coh}}(X)\leq (\dim X + 1) \cdot \operatorname{level} ^{\mathbb{R}\pi_\ast G} (\mathcal{O}_X). 
    \end{displaymath}
    Appealing to Example~\ref{ex:resolution_char_zero}, there are similar bounds for when $X$ is a Noetherian quasiexcellent integral affine scheme of finite Krull dimension.
\end{remark}

\begin{example}\label{ex:splinter_rouquier_bound}
    Recall a Noetherian scheme $X$ is said to be a \textbf{derived splinter} if for each proper surjective morphism $\pi\colon Y \to X$, the natural map $\mathcal{O}_X \to \mathbb{R}\pi_\ast \mathcal{O}_Y$ splits in $D_{\operatorname{Qcoh}}(X)$. If $X$ has prime characteristic, then this coincides with a being a splinter  \cite{Bhatt:2012}. Suppose $X=\operatorname{Spec}(R)$ is an affine derived splinter variety over a field $k$. If $\pi\colon \widetilde{X} \to X$ is a resolution of singularities, then $\pi$ is proper surjective. This means the natural map $\mathcal{O}_X \to \mathbb{R}\pi_\ast \mathcal{O}_{\widetilde{X}}$ splits in $D_{\operatorname{Qcoh}}(X)$, and hence, $\mathcal{O}_X \in \langle \mathbb{R}\pi_\ast \mathcal{O}_{\widetilde{X}}\rangle_1$. Choose $G\in D^b_{\operatorname{coh}}(\widetilde{X})$ a strong generator with minimal generation time. Then $G\bigoplus \mathcal{O}_{\widetilde{X}}$ is a strong generator for $D^b_{\operatorname{coh}}(\widetilde{X})$. By Corollary~\ref{cor:rouquir_dimension_bounds_affine_variety} and Remark~\ref{rmk:effective_bound_rouquier_dimension_affine_variety}, $\dim D^b_{\operatorname{coh}}(X)$ is at most $\dim D^b_{\operatorname{coh}}(\widetilde{X}) + 1$.
\end{example}

\begin{proof}[Proof of Proposition~\ref{prop:descend_bounds}]
    By Theorem~\ref{thm:descent_conditions}, $\mathbb{R}\pi_\ast G$ is a strong $\bigoplus$-generator in $D_{\operatorname{Qcoh}}(X)$. Furthermore, $\pi$ being proper guarantees $\mathbb{R}\pi_\ast G\in D^b_{\operatorname{coh}}(X)$, and so $\mathbb{R}\pi_\ast G$ is a strong generator for $D^b_{\operatorname{coh}}(X)$ via Remark~\ref{rmk:strong_oplus_to_strong_gerator_time_bound}. If $l:=\operatorname{level}^{\mathbb{R}\pi_\ast \mathcal{O}_Y}  (\mathcal{O}_X)$, then the proof of Theorem~\ref{thm:descent_conditions} ensures $ D_{\operatorname{Qcoh}}(X) = \overline{\langle \mathbb{R}\pi_\ast G\rangle}_{ln}$, and once more by Remark~\ref{rmk:strong_oplus_to_strong_gerator_time_bound}, $D^b_{\operatorname{coh}}(X) = \langle G \rangle_{nl}$.
\end{proof}

\begin{proof}[Proof of Corollary~\ref{cor:rouquir_dimension_bounds_affine_variety}]
    Denote by $g$ for the generation time of $G$ in $D^b_{\operatorname{coh}}(\widetilde{X})$. If $\pi$ is a proper surjective morphism, then Theorem~\ref{thm:proper_surjective_descent} ensures $\mathbb{R}\pi_\ast G$ is a strong generator in $D^b_{\operatorname{coh}}(X)$. Choose an object $E_0\in D^b_{\operatorname{coh}}(X)$. Consider a $c$-fold $\mathbb{R}\pi_\ast G$-coghost map $f\colon E_c \to E_0$. If it had been the case that $E_0 \in \langle \mathbb{R}\pi_\ast G\rangle_1$, then there exist is nothing to check, so without loss of generality assume $c>1$. The morphism $f\colon E_c \to E_0$ is a composition of $c$ maps which are $\mathbb{R}\pi_\ast G$-coghost, i.e.
    \begin{displaymath}
        E_c \xrightarrow{f_c} E_{c-1} \xrightarrow{f_{c-1}} \cdots \xrightarrow{f_2} E_1 \xrightarrow{f_1} E_0.
    \end{displaymath}
    However, there exists perfect complexes $P_j\in \operatorname{perf}X$ and a commutative diagram
    \begin{displaymath}
        \begin{tikzcd}
            {E_c} & {E_{c-1}} & \cdots & {E_1} & {E_0} \\
            {P_c} & {P_{c-1}} & \cdots & {P_1} & {E_0}
            \arrow["{f_c}", from=1-1, to=1-2]
            \arrow["{f_{c-1}}", from=1-2, to=1-3]
            \arrow["{f_2}", from=1-3, to=1-4]
            \arrow["{f_1}", from=1-4, to=1-5]
            \arrow["{}", from=2-1, to=1-1]
            \arrow["{g_c}"', from=2-1, to=2-2]
            \arrow["{}", from=2-2, to=1-2]
            \arrow["{g_{c-1}}"', from=2-2, to=2-3]
            \arrow["{g_2}"', from=2-3, to=2-4]
            \arrow["{}", from=2-4, to=1-4]
            \arrow["{g_1}"', from=2-4, to=2-5]
            \arrow["{=}"', from=2-5, to=1-5]
        \end{tikzcd}
    \end{displaymath}
    where each $g_j$ is $\mathbb{R}\pi_\ast G$-coghost \cite[$\S 2.8$]{Letz:2021}. Furthermore, this diagram enjoys the condition that $f_c\circ \cdots \circ f_1=0$ if, and only if, $g_c\circ \cdots \circ g_1=0$. If $n:=\operatorname{level} ^{\mathbb{R}\pi_\ast G}(\mathcal{O}_X)$, then $\mathcal{O}_X \in \langle \mathbb{R}\pi_\ast \mathcal{O}_X\rangle_n$, and so for all $P\in \operatorname{perf}X$,
    \begin{displaymath}
        P \in \langle \mathbb{R}\pi_\ast \mathcal{O}_X\rangle_n \overset{\mathbb{L}}{\otimes} P \subseteq \langle \mathbb{R}(G \overset{\mathbb{L}}{\otimes} \mathbb{L} \pi^\ast P)\rangle_n.
    \end{displaymath}
    This last inclusion comes from the fact tensoring with a perfect complex is an endofunctor on $D^b_{\operatorname{coh}}(X)$ and uses the projection formula. If $P\in \operatorname{perf}X$, then $G \overset{\mathbb{L}}{\otimes} \mathbb{L} \pi^\ast P \in D^b_{\operatorname{coh}}(\widetilde{X})$, and hence,
    \begin{displaymath}
        \begin{aligned}
            \operatorname{level}^{\mathbb{R}\pi_\ast G}  (P) &\leq \operatorname{level}^{\mathbb{R}\pi_\ast G}  ( \mathbb{R}(G \overset{\mathbb{L}}{\otimes} \mathbb{L} \pi^\ast P) ) \cdot \operatorname{level}^{\mathbb{R}(G \overset{\mathbb{L}}{\otimes} \mathbb{L} \pi^\ast P) }  (P)
            \\&\leq (g + 1) \cdot n. 
        \end{aligned}
    \end{displaymath}
    This ensures $\operatorname{perf}X \subseteq \langle \mathbb{R}\pi_\ast G\rangle_{n(g+1)}$. Hence, the composition
    \begin{displaymath}
        P_c \xrightarrow{g_c} P_{c-1} \xrightarrow{g_{c-1}} \cdots \xrightarrow{g_2} P_1
    \end{displaymath}
    vanishes when $c-1 \geq n (g+1)$. This ensures $f_c \circ \cdots \circ f_1=0$ when $c\geq n (g+1) + 1$. By \cite[Theorem 24]{OS:2012}, $E \in \langle \mathbb{R}\pi_\ast G \rangle_{n(g+1)+1}$, and this furnishes the proof.
\end{proof}

\bibliographystyle{alpha}
\bibliography{mainbib}

\newcommand{\etalchar}[1]{$^{#1}$}
\begin{thebibliography}{KMVdB11}

\bibitem[ABIM10]{ABIM:2010}
Luchezar~L. Avramov, Ragnar-Olaf Buchweitz, Srikanth~B. Iyengar, and Claudia
  Miller.
\newblock Homology of perfect complexes.
\newblock {\em Adv. Math.}, 223(5):1731--1781, 2010.

\bibitem[AO13]{Alexeev/Orlov:2013}
Valery Alexeev and Dmitri Orlov.
\newblock Derived categories of {Burniat} surfaces and exceptional collections.
\newblock {\em Math. Ann.}, 357(2):743--759, 2013.

\bibitem[Aok21]{Aoki:2021}
Ko~Aoki.
\newblock Quasiexcellence implies strong generation.
\newblock {\em J. Reine Angew. Math.}, 780:133--138, 2021.

\bibitem[BD11]{Burban/Drozd:2011}
Igor Burban and Yuriy Drozd.
\newblock Tilting on non-commutative rational projective curves.
\newblock {\em Math. Ann.}, 351(3):665--709, 2011.

\bibitem[BDS16]{Balmer/Dell'Ambrogio/Sanders:2016}
Paul Balmer, Ivo Dell'Ambrogio, and Beren Sanders.
\newblock Grothendieck-{Neeman} duality and the {Wirthm{\"u}ller} isomorphism.
\newblock {\em Compos. Math.}, 152(8):1740--1776, 2016.

\bibitem[BF12]{Ballard/Favero:2012}
Matthew Ballard and David Favero.
\newblock Hochschild dimensions of tilting objects.
\newblock {\em Int. Math. Res. Not.}, 2012(11):2607--2645, 2012.

\bibitem[BFK14]{Ballard/Favero/Katzarkov:2014}
Matthew Ballard, David Favero, and Ludmil Katzarkov.
\newblock A category of kernels for equivariant factorizations. {II}: {Further}
  implications.
\newblock {\em J. Math. Pures Appl. (9)}, 102(4):702--757, 2014.

\bibitem[BFK19]{Ballard/Favero/Katzarkov:2019}
Matthew Ballard, David Favero, and Ludmil Katzarkov.
\newblock Variation of geometric invariant theory quotients and derived
  categories.
\newblock {\em J. Reine Angew. Math.}, 746:235--303, 2019.

\bibitem[Bha12]{Bhatt:2012}
Bhargav Bhatt.
\newblock Derived splinters in positive characteristic.
\newblock {\em Compos. Math.}, 148(6):1757--1786, 2012.

\bibitem[BIL{\etalchar{+}}23]{BILMP:2023}
Matthew~R Ballard, Srikanth~B Iyengar, Pat Lank, Alapan Mukhopadhyay, and Josh
  Pollitz.
\newblock High frobenius pushforwards generate the bounded derived category.
\newblock {\em arXiv preprint arXiv:2303.18085}, 2023.

\bibitem[BKS18]{Bondal/Kap/Schechtman:2018}
Alexey Bondal, Mikhail Kapranov, and Vadim Schechtman.
\newblock Perverse schobers and birational geometry.
\newblock {\em Sel. Math., New Ser.}, 24(1):85--143, 2018.

\bibitem[BO95]{Bondal/Orlov:1995}
Alexei Bondal and Dmitri Orlov.
\newblock Semiorthogonal decomposition for algebraic varieties.
\newblock {\em arXiv preprint alg-geom/9506012}, 1995.

\bibitem[BO02]{Bondal/Orlov:2002}
A.~Bondal and D.~Orlov.
\newblock Derived categories of coherent sheaves.
\newblock In {\em Proceedings of the international congress of mathematicians,
  ICM 2002, Beijing, China, August 20--28, 2002. Vol. II: Invited lectures},
  pages 47--56. Beijing: Higher Education Press; Singapore: World
  Scientific/distributor, 2002.

\bibitem[Bri02]{Bridgeland:2002}
Tom Bridgeland.
\newblock Flops and derived categories.
\newblock {\em Invent. Math.}, 147(3):613--632, 2002.

\bibitem[Bri07]{Bridgeland:2007}
Tom Bridgeland.
\newblock Stability conditions on triangulated categories.
\newblock {\em Ann. Math. (2)}, 166(2):317--345, 2007.

\bibitem[BvdB03]{BVdB:2003}
A.~Bondal and M.~van~den Bergh.
\newblock Generators and representability of functors in commutative and
  noncommutative geometry.
\newblock {\em Mosc. Math. J.}, 3(1):1--36, 2003.

\bibitem[Chr98]{Christensen:1998}
J.~Daniel Christensen.
\newblock Ideals in triangulated categories: {Phantoms}, ghosts and skeleta.
\newblock {\em Adv. Math.}, 136(2):284--339, 1998.

\bibitem[DL24]{Dey/Lank:2024}
Souvik Dey and Pat Lank.
\newblock D{\'e}vissage for generation in derived categories.
\newblock {\em arXiv preprint arXiv:2401.13661}, 2024.

\bibitem[Efi20]{Efimov:2020}
Alexander~I. Efimov.
\newblock Homotopy finiteness of some {DG} categories from algebraic geometry.
\newblock {\em J. Eur. Math. Soc. (JEMS)}, 22(9):2879--2942, 2020.

\bibitem[ELS20]{Elagin/Lunts/Schnurer:2020}
Alexey Elagin, Valery~A. Lunts, and Olaf~M. Schn{\"u}rer.
\newblock Smoothness of derived categories of algebras.
\newblock {\em Mosc. Math. J.}, 20(2):277--309, 2020.

\bibitem[FR70]{Ferrand/Raynaud:1970}
Daniel Ferrand and Michel Raynaud.
\newblock Fibres formelles d'un anneau local noeth{\'e}rien.
\newblock {\em Ann. Sci. {\'E}c. Norm. Sup{\'e}r. (4)}, 3:295--311, 1970.

\bibitem[Gro60]{Grothendieck:1960}
A.~Grothendieck.
\newblock {\'E}l{\'e}ments de g{\'e}om{\'e}trie alg{\'e}brique. {I}: {Le}
  langage des sch{\'e}mas. {II}: {\'E}tude globale {\'e}l{\'e}mentaire de
  quelques classe de morphismes. {III}: {\'E}tude cohomologique des faisceaux
  coh{\'e}rents (premi{\`e}re partie).
\newblock {\em Publ. Math., Inst. Hautes {\'E}tud. Sci.}, 4:1--228, 1960.

\bibitem[HHL23]{Hanlon/Hicks/Lazarev:2023}
Andrew Hanlon, Jeff Hicks, and Oleg Lazarev.
\newblock Resolutions of toric subvarieties by line bundles and applications.
\newblock {\em arXiv preprint arXiv:2303.03763}, 2023.

\bibitem[Hir64]{Hironaka:1964a}
H.~Hironaka.
\newblock Resolution of singularities of an algebraic variety over a field of
  characteristic zero. {I}.
\newblock {\em Ann. Math. (2)}, 79:109--203, 1964.

\bibitem[Hoc73]{Hochster:1973}
M.~Hochster.
\newblock Non-openness of loci in {Noetherian} rings.
\newblock {\em Duke Math. J.}, 40:215--219, 1973.

\bibitem[HPS97]{Hovey/Palmieri/Strickland:1997}
Mark Hovey, John~H. Palmieri, and Neil~P. Strickland.
\newblock {\em Axiomatic stable homotopy theory}, volume 610 of {\em Mem. Am.
  Math. Soc.}
\newblock Providence, RI: American Mathematical Society (AMS), 1997.

\bibitem[Huy06]{Huybrechts:2006}
D.~Huybrechts.
\newblock {\em Fourier-{Mukai} transforms in algebraic geometry}.
\newblock Oxford Math. Monogr. Oxford: Clarendon Press, 2006.

\bibitem[IT16]{Iyengar/Takahashi:2016}
Srikanth~B. Iyengar and Ryo Takahashi.
\newblock Annihilation of cohomology and strong generation of module
  categories.
\newblock {\em Int. Math. Res. Not.}, 2016(2):499--535, 2016.

\bibitem[IT19]{Iyengar/Takahashi:2019}
Srikanth~B. Iyengar and Ryo Takahashi.
\newblock Openness of the regular locus and generators for module categories.
\newblock {\em Acta Math. Vietnam.}, 44(1):207--212, 2019.

\bibitem[Kaw06]{Kawamata:2006}
Yujiro Kawamata.
\newblock Derived categories of toric varieties.
\newblock {\em Mich. Math. J.}, 54(3):517--535, 2006.

\bibitem[Kaw22]{Kawamata:2022}
Yujiro Kawamata.
\newblock Semi-orthogonal decomposition of a derived category of a 3-fold with
  an ordinary double point.
\newblock {\em Recent Developments in Algebraic Geometry: To Miles Reid for his
  70th Birthday}, 478:183, 2022.

\bibitem[KMVdB11]{KMVdB:2011}
Bernhard Keller, Daniel Murfet, and Michel Van~den Bergh.
\newblock On two examples by {Iyama} and {Yoshino}.
\newblock {\em Compos. Math.}, 147(2):591--612, 2011.

\bibitem[KP21]{Kuznetsov/Perry:2021}
Alexander Kuznetsov and Alexander Perry.
\newblock Homological projective duality for quadrics.
\newblock {\em J. Algebr. Geom.}, 30(3):457--476, 2021.

\bibitem[Kra22]{Krause:2022}
Henning Krause.
\newblock {\em Homological theory of representations}, volume 195 of {\em Camb.
  Stud. Adv. Math.}
\newblock Cambridge: Cambridge University Press, 2022.

\bibitem[Kuz07]{Kuznetsov:2007}
Alexander Kuznetsov.
\newblock Homological projective duality.
\newblock {\em Publ. Math., Inst. Hautes {\'E}tud. Sci.}, 105:157--220, 2007.

\bibitem[Kuz21]{Kuznetsov:2021}
Alexander Kuznetsov.
\newblock Derived categories of families of sextic del {Pezzo} surfaces.
\newblock {\em Int. Math. Res. Not.}, 2021(12):9262--9339, 2021.

\bibitem[Let21]{Letz:2021}
Janina~C. Letz.
\newblock Local to global principles for generation time over commutative
  {Noetherian} rings.
\newblock {\em Homology Homotopy Appl.}, 23(2):165--182, 2021.

\bibitem[LO24]{Lank/Olander:2024}
Pat Lank and Noah Olander.
\newblock Approximation by perfect complexes detects {Rouquier} dimension.
\newblock {\em arXiv preprint arXiv:2401.10146}, 2024.

\bibitem[Lun10]{Lunts:2010}
Valery~A. Lunts.
\newblock Categorical resolution of singularities.
\newblock {\em J. Algebra}, 323(10):2977--3003, 2010.

\bibitem[Mat89]{Matsumura:1989}
Hideyuki Matsumura.
\newblock {\em Commutative ring theory}, volume~8 of {\em Cambridge Studies in
  Advanced Mathematics}.
\newblock Cambridge University Press, Cambridge, second edition, 1989.
\newblock Translated from the Japanese by M. Reid.

\bibitem[Nee92]{Neeman:1992}
Amnon Neeman.
\newblock The chromatic tower for {{\(D(R)\)}}. {With} an appendix by {Marcel}
  {B{\"o}kstedt}.
\newblock {\em Topology}, 31(3):519--532, 1992.

\bibitem[Nee21]{Neeman:2021}
Amnon Neeman.
\newblock Strong generators in {{\(\mathbf{D}^{\mathrm{perf}}(X)\)}} and
  {{\(\mathbf{D}^b_{\mathrm{coh}}(X)\)}}.
\newblock {\em Ann. Math. (2)}, 193(3):689--732, 2021.

\bibitem[Oka11]{Okawa:2011}
Shinnosuke Okawa.
\newblock Semi-orthogonal decomposability of the derived category of a curve.
\newblock {\em Adv. Math.}, 228(5):2869--2873, 2011.

\bibitem[Orl05]{Orlov:2005}
D.~O. Orlov.
\newblock Derived categories of coherent sheaves and motives.
\newblock {\em Russ. Math. Surv.}, 60(6):1242--1244, 2005.

\bibitem[Orl09]{Orlov:2009}
Dmitri Orlov.
\newblock Remarks on generators and dimensions of triangulated categories.
\newblock {\em Mosc. Math. J.}, 9(1):143--149, 2009.

\bibitem[O{\v{S}}12]{OS:2012}
Steffen Oppermann and Jan {\v{S}}t'ov{\'{\i}}{\v{c}}ek.
\newblock Generating the bounded derived category and perfect ghosts.
\newblock {\em Bull. Lond. Math. Soc.}, 44(2):285--298, 2012.

\bibitem[Pol19]{Pollitz:2019}
Josh Pollitz.
\newblock The derived category of a locally complete intersection ring.
\newblock {\em Adv. Math.}, 354:18, 2019.
\newblock Id/No 106752.

\bibitem[PPZ22]{Perry/Pertusi/Zhao:2022}
Alexander Perry, Laura Pertusi, and Xiaolei Zhao.
\newblock Stability conditions and moduli spaces for {Kuznetsov} components of
  {Gushel}-{Mukai} varieties.
\newblock {\em Geom. Topol.}, 26(7):3055--3121, 2022.

\bibitem[PS21]{Pavic/Shinder:2021}
Nebojsa Pavic and Evgeny Shinder.
\newblock {{\(K\)}}-theory and the singularity category of quotient
  singularities.
\newblock {\em Ann. \(K\)-Theory}, 6(3):381--424, 2021.

\bibitem[Rou08]{Rouquier:2008}
Rapha{\"e}l Rouquier.
\newblock Dimensions of triangulated categories.
\newblock {\em J. \(K\)-Theory}, 1(2):193--256, 2008.

\bibitem[Rou10]{Rouquier:2010}
Rapha{\"e}l Rouquier.
\newblock Derived categories and algebraic geometry.
\newblock In {\em Triangulated categories. Based on a workshop, Leeds, UK,
  August 2006}, pages 351--370. Cambridge: Cambridge University Press, 2010.

\bibitem[Tem08]{Temkin:2008}
Michael Temkin.
\newblock Desingularization of quasi-excellent schemes in characteristic zero.
\newblock {\em Adv. Math.}, 219(2):488--522, 2008.

\bibitem[Ver96]{Verdier:1996}
Jean-Louis Verdier.
\newblock {\em Des cat{\'e}gories d{\'e}riv{\'e}es des cat{\'e}gories
  ab{\'e}liennes}, volume 239 of {\em Ast{\'e}risque}.
\newblock Paris: Soci{\'e}t{\'e} Math{\'e}matique de France, 1996.

\end{thebibliography}

\end{document}